\documentclass[10pt]{amsart}
\usepackage{amsfonts,amscd,amsthm,amsgen,amsmath,amssymb}
\usepackage[all]{xy}
\usepackage[vcentermath]{youngtab}
\newtheorem{theorem}{Theorem}[section]
\newtheorem{lemma}[theorem]{Lemma}
\newtheorem{corollary}[theorem]{Corollary}
\newtheorem{proposition}[theorem]{Proposition}

\theoremstyle{definition}
\newtheorem{definition}[theorem]{Definition}
\newtheorem{example}[theorem]{Example}

\theoremstyle{remark}
\newtheorem{remark}[theorem]{Remark}

\numberwithin{equation}{section}

\begin{document}

\title[Special K\"ahler geometry and topological recursion]{Special K\"ahler geometry of the Hitchin system and topological recursion}
\author{David Baraglia}
\author{Zhenxi Huang}

\address{School of Mathematical Sciences, The University of Adelaide, Adelaide SA 5005, Australia}

\email{david.baraglia@adelaide.edu.au}
\email{zhenxi.huang@adelaide.edu.au}

\begin{abstract}
We investigate the special K\"ahler geometry of the base of the Hitchin integrable system in terms of spectral curves and topological recursion. The Taylor expansion of the special K\"ahler metric about any point in the base may be computed by integrating the $g = 0$ Eynard-Orantin invariants of the corresponding spectral curve over cycles. In particular, we show that the Donagi-Markman cubic is computed by the invariant $W^{(0)}_3$. We use topological recursion to go one step beyond this and compute the symmetric quartic of second derivatives of the period matrix.

\end{abstract}
\thanks{D. Baraglia is financially supported by the Australian Research Council Discovery Early Career Researcher Award DE160100024.}


\date{\today}



\maketitle


\section{Introduction}\label{sec:intro}

The Hitchin integrable system \cite{hit1,hit3} ties together many seemingly different branches of geometry and physics, including twistor theory, integrable systems, mirror symmetry and supersymmetric Yang-Mills theory to name just a few. The goal of this paper is to elucidate one particular aspect of this rich story, namely the relation between the special K\"ahler geometry of the base of the Hitchin system with the theory of Eynard-Orantin topological recursion of the spectral curves.

Suppose that $f : \mathcal{M} \to B$ is a complex integrable system, so $\mathcal{M}$ is a complex manifold with a holomorphic symplectic form and $f$ is a complex Lagrangian fibration, whose generic fibres are complex tori. In general $f$ will have singular fibres so let $f : \mathcal{M}^{\rm reg} \to B^{\rm reg}$ denote the regular locus consisting of the non-singular torus fibres of $f$. As we recall in \textsection \ref{sec:complexlag}, under mild assumptions, there is a naturally defined metric $g^{\rm sk}$ on the base $B^{\rm reg}$, known as a {\em special K\"ahler metric} \cite{dw,Freed,Hit2}. We recall the fundamentals of special K\"ahler geometry in \textsection \ref{sec:skg}. In particular, $B^{\rm reg}$ has a K\"ahler metric of the form
\begin{equation}\label{equ:skmetric}
\omega = -\frac{i}{2} Im( \tau_{ij} ) dz^i \wedge d\overline{z}^j
\end{equation}
where $\tau_{ij}$ is a matrix of functions, the periods of the torus fibres of the integrable system. It is well known that the special K\"ahler metric on $B^{\rm reg}$ can be combined with a metric along the fibres to produce a hyperk\"ahler metric $g^{\rm sf}$ on $\mathcal{M}^{\rm reg}$, known as the {\em semi-flat hyperk\"ahler metric} \cite{cfg,Freed,Hit2}. This metric is called ``semi-flat" because its restriction to the fibres of $f : \mathcal{M}^{\rm reg} \to B^{\rm reg}$ is flat. 

In this paper we concentrate on the case that $f : \mathcal{M} \to B$ is the Hitchin system. Then $\mathcal{M} = \mathcal{M}_{n,d}$ is the {\em moduli space of Higgs bundles} of given rank $n$ and degree $d$. In this case $\mathcal{M}$ is known to admit a complete hyperk\"ahler metric $g$ \cite{hit1}. The semi-flat metric on $\mathcal{M}^{\rm reg}$ may be thought of as an approximation of the complete hyperk\"ahler metric. The semi-flat metric fails to extend over the singular fibres, however it is expected that $g$ can be recovered from $g^{\rm sf}$ by incorporating instanton corrections \cite{gmn}. 

To define the Hitchin system, one takes a compact Riemann surface $\Sigma$ of genus greater than $1$ and $\mathcal{M}$ to be the moduli space of semistable Higgs bundles of fixed rank $n$ and degree $d$ (see \textsection \ref{sec:higgsb}). In this case, the period matrices $\tau_{ij}(b)$ for $b \in B^{\rm reg}$ are not arbitrary, in fact they are the periods of a Riemann surface $S_b$, the {\em spectral curve} associated to the point $b \in B^{\rm reg}$. The spectral curve $S_b$ is a smooth compact Riemann surface $S_b \subset T^*\Sigma$ embedded in the cotangent bundle of $\Sigma$, such that the projection $\pi : S_b \to \Sigma$ is a branched covering of degree $n$. We now give a brief summary of the special K\"ahler geometry on $B^{\rm reg}$ in terms of spectral curves. Let $\theta$ be the canonical $1$-form on $T^*\Sigma$. A pair of local holomorphic coordinate systems $(z^1 , \dots , z^{g_S})$ and $(w_1 , \dots , w_{g_S})$ for $B^{\rm reg}$ are given by integrating the canonical $1$-form over a basis of $1$-cycles in $S_b$:
\[
z^i = \int_{a_i} \theta, \quad \quad w_i = \int_{b_i} \theta,
\]
where $a_1 , \dots , a_{g_S} , b_1 , \dots , b_{g_S}$ is a symplectic basis of $1$-cycles, which may be defined over any sufficiently small neighbourhood in $B^{\rm reg}$. The monodromy of the Hitchin system prevents us from choosing such a basis globally on $B^{\rm reg}$. Having chosen such a basis of $1$-cycles, we have the normalised basis of holomorphic $1$-forms $\omega_1 , \dots , \omega_{g_S}$ characterised by $\int_{a_i} \omega_j = \delta_{ij}$. The period matrix of $S_b$ is then given by $\tau_{ij} = \int_{b_i} \omega_j$. With respect to the coordinate system $(z^1 , \dots , z^{g_S})$, the special K\"ahler metric is given by Equation (\ref{equ:skmetric}). Alternatively, we have a real coordinate system $(x^1 , \dots , x^{g_S} , y_1 , \dots , y_{g_S})$ given by:
\[
x^i = Re(z^i) = Re\left( \int_{a_i} \theta \right), \quad \quad y_i = Re(w_i) = Re \left(\int_{b_i} \theta \right).
\]
These coordinates are globally defined up to monodromy, which acts by linear transformations, hence they define an affine structure on $B^{\rm reg}$. Moreover they are Darboux coordinates for the K\"ahler form:
\[
\omega = dx^1 \wedge dy_1 + \dots + dx^{g_S} \wedge dy_{g_S}.
\]
Recall the {\em prepotential} (see \textsection \ref{sec:skg}) is a locally defined holomorphic function $\mathcal{F}$ on $B^{\rm reg}$ such that $w_i = \frac{\partial \mathcal{F} }{\partial z^i}$. We will recall that the period matrix is given by $\tau_{ij} = \frac{\partial w_i}{\partial z^j}$ and thus 
\[
\tau_{ij} = \frac{\partial^2 \mathcal{F}}{\partial z^i \partial z^j}.
\]
From the prepotential $\mathcal{F}$, one obtains a K\"ahler potential $K$ by:
\begin{equation*}
K = -\frac{1}{2} Im \left( \frac{\partial \mathcal{F}}{\partial z^i} \overline{z}^i \right).
\end{equation*}

An important quantity in special K\"ahler geometry is the symmetric cubic $c \in H^0( \mathcal{M} , Sym^3( T_\mathcal{M} ) )$ which measures the variation of the period matrix $\tau_{ij}$:
\[
c = c_{ijk} dz^i \otimes dz^j \otimes dz^k, \quad \quad c_{ijk} = \frac{\partial \tau_{jk}}{\partial z^i} = \frac{\partial^3 \mathcal{F} }{\partial z^i \partial z^j \partial z^k}.
\]
We call $c$ the {\em Donagi-Markman cubic}, since in the case where $\mathcal{M}$ is the base of a complex integrable system, $c$ is the cubic studied by Donagi and Markman \cite{dm}. One can for instance express the Riemann, Ricci and scalar curvatures in terms of $\tau_{ij}$ and $c_{ijk}$ \cite{Freed}. The following proposition (cf. \cite{hhp} for similar results) summarises the relation between the special K\"ahler geometry on $B^{\rm reg}$ and the spectral curves:
\begin{proposition}
\begin{itemize}
We have the following relations:
\item[(1)]{A (local) prepotential for the special K\"ahler structure on $B^{\rm reg}$ is given by:
\[
\mathcal{F} = \frac{1}{2} z^i w_i = \frac{1}{2} \tau_{ij} z^i z^j.
\]
}
\item[(2)]{A (global) K\"ahler potential is given by:
\[
K = -\frac{i}{4} \int_{S} \theta \wedge \overline{\theta}.
\]
}
\item[(3)]{Let $\mathcal{V}_{\mathbb{C}}$ be the local system on $B^{\rm reg}$ whose fibre over $b$ is $H^1( S_b , \mathbb{C})$. We may think of $\theta$ as a section of $\mathcal{V}_{\mathbb{C}}$. The Donagi-Markman cubic is given by the ``Yukawa couplings":
\[
c(X,Y,Z) = \int_S \nabla_X \nabla_Y \nabla_Z \theta \wedge \theta,
\]
where $\nabla$ is the Gauss-Manin connection.
}
\end{itemize}
\end{proposition}

The above formula for the Donagi-Markman cubic involves differentiation with respect to the Gauss-Manin connection. This requires knowledge of the family of spectral curves. In \textsection \ref{sec:dmcubic} we give another formula for the $c$ in terms of the geometry of a single spectral curve. Our formula holds for any smooth spectral curve, but in this section we will for simplicity consider the case where $\pi : S \to \Sigma$ has only simple branching. Around each ramification point $a \in S$ we can thus find local coordinates $x$ on $\Sigma$ and $q$ on $S$ such that $\pi(q) = x = q^2$. Near $a$ we have $\theta = y dx$ for some function $y(q)$ with $y'(a) \neq 0$. We have:
\begin{theorem}\label{thm:dmcub0}
The Donagi-Markman cubic is given by:
\begin{equation*}
c_{ijk} = -2\pi i \sum_a \underset{a}{Res} \; \left( \frac{ \omega_i \omega_j \omega_k }{dxdy} \right),
\end{equation*}
where the sum is over the ramification points of $\pi$.
\end{theorem}
Similar-looking formulas for the Donagi-Markman cubic have appeared in \cite{bal,hhp}, however these formulas use cameral curves instead of spectral curves and involve quadratic residues instead of ordinary residues. One advantage of our formula (in the form of Theorem \ref{thm:dmcub}) is that it applies even when $\pi$ has higher order ramification. Moreover, our formula looks very similar to formulas appearing in the theory of Eynard-Orantin topological recursion. This is not a coincidence, as we show in Section \textsection \ref{sec:toprec}. 

\subsection{Relation to topological recursion}

Eynard-Orantin topological recursion \cite{EO} is a recursive formula which takes as input a Riemann surface $S$ (which we assume is compact) with a pair of meromorphic functions $x,y$ and produces a series of symmetric multidifferentials $W^{(g)}_n$, for $g \ge 0$, $n \ge 1$, the {\em Eynard-Orantin invariants}. More precisely, $W^{(g)}_n$ is a meromorphic section of the $n$-th exterior tensor product $K_S^{\boxtimes n} = K_S \boxtimes K_S \boxtimes \cdots \boxtimes K_S$ on $S^n$, where $K_S$ denotes the canonical bundle of $S$, which is symmetric under interchange of factors. The function $x$, viewed as a map $x : S \to \mathbb{P}^1$ is assumed to be a branched covering with only simple branching. The topological recursion formula has been extended to Hitchin spectral curves $S \subset T^*\Sigma$ in \cite{dumu}, which gives an interpretation of the Eynard-Orantin invariants in terms of quantisation of spectral curves. Our paper gives another interpretation of these invariants, at least in the case of $g=0$. To make sense of the Eyndard-Orantin invariants in this setting, first note that the projection $\pi : S \to \Sigma$ plays the role of $x$. As we will recall in \textsection \ref{sec:toprec}, the topological recursion formula continues to makes sense as long as $\pi$ has only simple branching. They key point is that the recursion formula does not directly involve the functions $x,y$ but only the $1$-form $y dx$. In the case of a spectral curve $S \subset T^*\Sigma$, the canonical $1$-form $\theta$ plays the role of $y dx$. We show that the variational formula for Eynard-Orantin invariants in \cite{EO} holds in our setting and leads to the following formula for derivatives of the period matrix $\tau_{ij}$ about any point $b \in B^{\rm reg}$.
\begin{theorem}\label{thm:varper0}
For any $b \in B^{\rm reg}$, we have:
\[
\partial_{i_1} \partial_{i_2} \cdots \partial_{i_{m-2}} \tau_{ i_{m-1} i_m }(b) = -\left( \frac{i}{2\pi} \right)^{m-1} \int_{p_1 \in b_{i_1}} \cdots \int_{p_m \in b_{i_m} } W^{(0)}_m(p_1 , \dots , p_m)
\]
where on the right hand side, $W^{(0)}_m(p_1 , \dots , p_m)$ is the $W^{(0)}_m$ Eynard-Orantin invariant of the corresponding spectral curve $S_b$. One can show that $W^{(0)}_m$ in each variable only has poles with zero residue, so the above expression does not depend on the choice of paths representing the given cycles.
\end{theorem}
This formula shows that the $g=0$ Eynard-Orantin invariants $W^{(0)}_k$ for a spectral curve $S_b$ compute the Taylor series expansion of the period matrix $\tau_{ij}$ about $b \in B^{\rm reg}$. Since the special K\"ahler metric on $B^{\rm reg}$ is given in terms of the period matrix, the invariants $W^{(0)}_k$ also compute the power series expansion of the special K\"ahler metric. For instance, we consider the $m=1$ case and show that it recovers our formula (Theorem \ref{thm:dmcub0}) for the Donagi-Markman cubic. Theorem \ref{thm:varper0} is remarkable in that the left hand side of the equation is related to geometry of the {\em family of spectral curves}, while the right hand side is given by invariants of a {\em single spectral curve} $S_b \subset T^*\Sigma$. In a sense, any single spectral curve $S_b$ ``knows" about the geometry of the entire family $\{ S_b \}_{b \in B^{\rm reg}}$.

Let us remark that Theorem \ref{thm:varper0} is certainly not surprising, as it is essentially a consequence of the well known variational formula for Eynard-Orantin invariants given in \cite{EO}. The main point we would like to emphasise is that this formula is applicable to Hitchin spectral curves. This is not immediately obvious as one needs to check that the proof of the variational formula in \cite{EO} holds for Hitchin spectral curves. Ultimately, this boils down to proving a version of the so-called ``Rauch variational formula", which we prove in Proposition \ref{prop:rvf}:
\begin{proposition}
Let $\partial \in T_b B^{\rm reg}$. Assume $p,r$ are distinct and are not ramification points. Then:
\[
\delta B( p, r ) = -\sum_a  \underset{u \to a}{Res} \; \frac{ \delta \theta(u) B(u,p)B(u,r) }{ dx(u) dy(u) },
\]
where the sum is over the ramification points of $\pi$ and for each ramification point $a \in S_b$, we choose coordinate functions $x$ on $\Sigma$ and $q$ on $S$ with $x = q^2$ and write $\theta = y dx$.
\end{proposition}
In this proposition, $B$ is the {\em Bergman kernel} (see \textsection \ref{sec:toprec}). For the meaning of the variational operator $\delta$, see Definition \ref{def:delta}. Versions of this formula are well known in the literature, but we could not find a proof that holds in the setting of Hitchin spectral curves, so we give a proof in \textsection \ref{sec:toprec}.

We finish the paper by using topological recursion to go beyond the Donagi-Markman cubic and compute the symmetric quartic of second derivatives of the periods in terms of geometry of the spectral curve. Around any ramification point $a \in S_b$, we may write the normalised $1$-forms as $\omega_i = \omega_i(q)dq$ for some functions $\omega_i(q)$. Our result is:
\begin{theorem}
The second derivatives of the period matrix are given by:
\begin{equation*}
\begin{aligned}
\partial_i \partial_j \tau_{kl} &= 2\pi i \sum_{a \neq b} B(a,b) \left( \frac{ \omega_i(a)\omega_j(a)}{2y'(a)} \right) \left( \frac{ \omega_k(b)\omega_l(b)}{2y'(b)} \right) + \text{cyc}_{j,k,l} \\
& \quad \quad + 2\pi i \sum_a  \frac{1}{8 y'(a)^2}\left( S_B(a) - \frac{y'''(a)}{y'(a)} \right) \omega_i(a)\omega_j(a)\omega_k(a)\omega_l(a) \\
& \quad \quad + 2 \pi i \sum_a \frac{1}{8y'(a)^2}\left( \omega_i''(a) \omega_j(a)\omega_k(a)\omega_l(a) \right) + \text{cyc}_{i,j,k,l},
\end{aligned}
\end{equation*}
where $S_B$ is the Bergman projective connection (see \textsection \ref{sec:toprec}), the sum $\sum_{a \neq b}$ is over distinct pairs of ramification points and $\text{cyc}$ means to sum over cyclic permutations of the specified indices.
\end{theorem}

\subsection{On the $g>0$ invariants}
In this paper, we have established the relation between the $g=0$ Eynard-Orantin invariants of spectral curves and the special K\"ahler geometry of $B^{\rm reg}$. An interesting problem would be to find a similar geometric interpretation for the $g > 0$ invariants. In particular one may ask whether there is a relation between the $g > 0$ invariants and the instanton corrections relating the semi-flat hyperk\"ahler metric on $\mathcal{M}^{\rm reg}$ to the complete hyperk\"ahler metric on $\mathcal{M}$.

\subsection{Summary of paper}
A brief summary of the contents of this paper is as follows. In \textsection \ref{sec:skg} we give a review of special K\"ahler geometry. In \textsection \ref{sec:complexlag} we recall a result of Hitchin showing that the moduli space of deformations of a compact complex Lagrangian in a complex symplectic manifold $M$ has a natural special K\"ahler metric (assuming $M$ admits a K\"ahler metric). In \textsection \ref{sec:higgsb} we briefly review the relevant aspects of the moduli space of Higgs bundles and the Hitchin system $f : \mathcal{M}_{n,d} \to B$, in particular the spectral curve construction. There are two possible ways in which $B^{\rm reg}$ inherits a special K\"ahler geometry: (i) view $B^{\rm reg}$ as a family of complex Lagrangians in $\mathcal{M}_{n,d}$ (the fibres of the map $f$), or (ii) view $B^{\rm reg}$ as a family of spectral curves $S_b \subset T^*\Sigma$ (clearly $S_b$ is a complex Lagrangian submanifold of $T^*\Sigma$). We show in \textsection \ref{sec:skhs} that these two points of view give rise to the same special K\"ahler geometry, and we describe this geometry in terms of the family of spectral curves. In \textsection \ref{sec:dmcubic}, we give a residue formula for the Donagi-Markman cubic, essentially by a computation of Kodaira-Spencer classes. In \textsection \ref{sec:toprec}, we consider the Eynard-Orantin invariants of Hitchin spectral curves and relate the $g=0$ invariants to the special K\"ahler geometry on $B^{\rm reg}$. We show that in this way, we recover our formula from \textsection \ref{sec:dmcubic} for the Donagi-Markman cubic and then proceed to compute the quartic of second derivatives of the periods by topological recursion.

\section{Special K\"ahler geometry}\label{sec:skg}

\subsection{Review of special K\"ahler geometry}

\begin{definition}[\cite{Freed,Hit2}]
A special K\"ahler manifold is a K\"ahler manifold $(M , g , I , \omega )$ together with a torsion free, flat affine connection $\nabla$ such that
\begin{itemize}
\item{$\nabla \omega = 0$, and}
\item{$d_\nabla I = 0$}
\end{itemize}
Here $I \in \Omega^1(M , TM)$ is viewed as a $TM$-valued $1$-form and $d_\nabla : \Omega^1(M , TM) \to \Omega^2(M , TM)$ is the differential induced by $\nabla$.
\end{definition}

Let us examine what the special K\"ahler condition implies in terms of local coordinates, following Freed \cite{Freed}. Since $\nabla$ is flat and torsion free, we can find local coordinates in which $\nabla$ becomes the trivial connection. Moreover, since $\nabla \omega = 0$, we can choose these coordinates to be Darboux, that is, $M$ has local flat coordinates $(x^1 , \dots , x^n , y_1 , \dots , y_n)$ for which
\[
\omega = dx^1 \wedge dy_1 + \dots + dx^n \wedge dy_n.
\]
Next, we observe that the $1$-forms $I dx^i$ are closed, because 
\[
d( I dx^i ) = d_\nabla( I dx^i ) = (d_\nabla I) \wedge dx^i + I d_\nabla( dx^i) = 0,
\]
where in the first equality we used that $\nabla$ is torsion free. It follows that locally there exist functions $u^1, \dots , u^n$ such that $Idx^i = du^i$. Let $z^i = x^i - iu^i$. Then $dz^i = dx^i - idu^i = dx^i - i I dx^i$ is a $(1,0)$-form. This together with the fact that $Re(dz^i) = dx^i$ implies that $(z^1 , \dots , z^n)$ is a local holomorphic coordinate system for $M$. Similarly, one can find functions $v_1, \dots , v_n$ such that $I dy_j = dv_j$ and setting $w_j = y_j - iv_j$ gives another holomorphic coordinate system $(w_1 , \dots , w_n)$\footnote{Note that Freed uses a slightly different convention in which $Re(w_i) = -y_i$.}. A simple computation gives
\[
\frac{\partial}{\partial z^i} = \frac{1}{2} \left( \frac{\partial}{\partial x^i} + \tau_{ij} \frac{\partial}{\partial y_j} \right), \quad \text{where } \;  \tau_{ij} = \frac{\partial w_j}{\partial z^i}.
\]
Compatibility of $\omega$ and $I$ gives, after a short computation, the condition $\tau_{ij} = \tau_{ji}$. So there is a local holomorphic function $\mathcal{F}$, called the {\em prepotential} such that
\[
w_i = \frac{\partial \mathcal{F}}{\partial z^i}, \quad \quad \tau_{ij} = \frac{\partial^2 \mathcal{F}}{\partial z^i \partial z^j}.
\]
From symmetry of $\tau_{ij}$, we also deduce that
\[
\omega = -\frac{i}{2} Im( \tau_{ij} ) dz^i \wedge d\overline{z}^j.
\]
If we use the convention that $g$ and $\omega$ are related by $g(X,Y) = \omega(IX,Y)$, this means that $\tau_{ij}$ is a symmetric, complex $n \times n$ matrix with $Im(\tau_{ij})$ positive definite. That is, $\tau_{ij}$ is a period matrix, a point in the Siegel upper half-space. We note that a K\"ahler potential for $\omega$ is given by:
\begin{equation}\label{equ:kahlerpot}
K = -\frac{1}{2} Im( w_i \overline{z}^i ) = -\frac{1}{2} Im \left( \frac{\partial \mathcal{F}}{\partial z^i} \overline{z}^i \right).
\end{equation}
As in the introduction, we have the Donagi-Markman cubic $c \in H^0( M , Sym^3( T_M ) )$ which measures the variation of the period matrix $\tau_{ij}$:
\[
c = c_{ijk} dz^i \otimes dz^j \otimes dz^k, \quad \quad c_{ijk} = \frac{\partial \tau_{jk}}{\partial z^i} = \frac{\partial^3 \mathcal{F} }{\partial z^i \partial z^j \partial z^k}.
\]

\subsection{Special K\"ahler manifolds as ``bi-Lagrangians"}

In \cite{Hit2}, Hitchin establishes a close relation between special K\"ahler manifolds and submanifolds of a complex symplectic vector space satisfying a ``bi-Lagrangian" condition. We recall the result. Let $V$ be a real symplectic vector space with symplectic form $\omega$. Define $V_{\mathbb{C}} = V \otimes_{\mathbb{R}} \mathbb{C} = V \oplus V$ with complex structure $I = \left( \begin{matrix} 0 & -1 \\ 1 & 0 \end{matrix} \right)$. Define a complex symplectic form $\omega^c = \omega_1 + i\omega_2$ on $V_{\mathbb{C}}$ as the $\mathbb{C}$-bilinear extension of $\omega$, that is:
\[
\omega^c( (x,y) , (x',y') ) = \omega( x+iy , x'+iy') = \underbrace{\left( \omega(x,x') - \omega(y,y') \right)}_{\omega_1( (x,y) , (x',y') )} + i\underbrace{\left( \omega(x,y') - \omega(x',y) \right)}_{\omega_2( (x,y) , (x',y') )}.
\]
Define in addition an (indefinite signature) inner product $g$ on $V_{\mathbb{C}}$ by
\[
g( (x,y) , (x',y') ) = \frac{1}{2}\left( \omega(x,y') + \omega(x',y) \right).
\]
Note that $g( \alpha , \beta ) = Re\left( \frac{i}{2} \omega^c( \alpha , \overline{\beta} ) \right)$, where we define $\overline{ (x,y) } = (x,-y)$.

\begin{theorem}[Hitchin, \cite{Hit2}]\label{thm:bilag}
Let $M \subset V_{\mathbb{C}}$ be a submanifold which is Lagrangian with respect to $\omega_1$ and $\omega_2$ and such that $g|_M$ is positive definite. Then $(M , g|_M , I|_M )$ is special K\"ahler. The projection of $M$ to the first factor $V \subset V_{\mathbb{C}}$ defines a system of local coordinates, and the flat affine connection $\nabla$ is the trivial connection on $TM$ with respect to these coordinates. In a similar manner, the symplectic form $\omega$ on $M$ is obtained by pullback of the symplectic form $\omega$ on $V$. Conversely, any special K\"ahler metric is locally of this form.
\end{theorem}

The relation between the local embedding $M \subset V_{\mathbb{C}}$ and the holomorphic coordinates $z^i , w_i$ is as follows: choose a symplectic bases $a^1 , \dots , a^n , b^1 , \dots , b^n$ for $V$, giving an explicit isomorphism $V \cong \mathbb{R}^{2n}$. Then the map $s : M \to V_{\mathbb{C}} = V \oplus V \cong (\mathbb{R}^{n})^4$ is given by
\[
( Re(z) , Re(w) , Im(z) , Im(w) ),
\]
where we think of $z = (z^1 , \dots , z^n)$, $w = (w_1 , \dots , w_n)$ as vectors in $\mathbb{C}^n$. Let $a_1 , \dots , a_n , b_1 , \dots , b_n \in V^*$ be the dual basis and $\langle \; , \; \rangle : V^* \otimes V \to \mathbb{R}$ the dual pairing, which we extend to a pairing $V^*_{\mathbb{C}} \otimes V_{\mathbb{C}} \to \mathbb{C}$ by $\mathbb{C}$-linearity. Then the coordinate systems $z^i , w_i$ can be recovered as:
\[
z^i(m) = \langle a_i , s(m) \rangle, \quad \quad w_i(m) = \langle b_i , s(m) \rangle.
\]

A slight extension of Theorem \ref{thm:bilag} is to consider the following situation: suppose $M$ is an $n$-manifold and let $(\mathcal{V} , \omega , \nabla)$ be a real symplectic vector bundle of rank $2n$ equipped with a flat symplectic connection $\nabla$. Let $\mathcal{V}_{\mathbb{C}} = \mathcal{V} \otimes_{\mathbb{R}} \mathbb{C}$ be the complexification and as above, define a complex symplectic form $\omega^c = \omega_1 + i\omega_2$ and an inner product $g$. Let $s : M \to \mathcal{V}_\mathbb{C}$ be a section of $\mathcal{V}_{\mathbb{C}}$ satisfying the following conditions:
\begin{itemize}
\item[(SK1)]{The bundle map $\rho : TM \to \mathcal{V}_{\mathbb{C}}$ given by $\rho(X) = \nabla_X s$ is injective.}
\item[(SK2)]{The image of $\rho$ is Lagrangian with respect to $\omega_1$ and $\omega_2$.}
\item[(SK3)]{The image of $\rho$ is positive definite with respect to $g$.}
\end{itemize}
Then $M$ inherits a special K\"ahler geometry. Indeed, locally on $M$ we choose a flat trivialisation $\mathcal{V} \cong M \times V$. Then $s$ defines an immersion $s : M \to V_\mathbb{C}$ and we are back to the setting of Theorem \ref{thm:bilag}.

\section{Relation to moduli spaces of complex Lagrangians}\label{sec:complexlag}

We recall the relationship between moduli spaces of complex Lagrangians and special K\"ahler geometry \cite{Hit2}.

\subsection{Deformations of complex Lagrangians}

Let $\mathcal{M}$ be a complex manifold of complex dimension $2n$ and let $\Omega$ be a holomorphic symplectic form, by which we mean a closed $(2,0)$-form such that $\wedge^n \Omega$ is non-vanishing. A {\em complex Lagrangian} in $\mathcal{M}$ is a complex submanifold $Y \subset \mathcal{M}$ which is Lagrangian with respect to $\Omega$, i.e. $Y$ has complex dimension $n$ and $\Omega|_Y = 0$. A real submanifold $Y \subset \mathcal{M}$ of real dimension $2n$ such that $\Omega|_Y = 0$ is in fact automatically a complex Lagrangian \cite[Proposition 1]{Hit2}. If $Y \subset \mathcal{M}$ is a complex Lagrangian, then $\Omega$ yields an isomorphism $N_Y \to T^*_Y$ between the normal bundle of $Y$ and the cotangent bundle, which sends a normal vector field $X$ to $ i_X \Omega |_Y$. In particular, this gives an isomorphism $H^0( Y , N_Y ) \cong H^1( Y , T^*_Y)$ between normal vector fields and holomorphic $1$-forms. Recall that $H^0(Y , N_Y)$ describes the space of infinitesimal deformations of $Y$ as a complex submanifold of $\mathcal{M}$. The infinitesimal deformations as a complex Lagrangian are those for which the corresponding holomorphic $1$-form is closed. 

Following Hitchin, we make the following two assumptions: (i) $Y$ is compact and (ii) $\mathcal{M}$ has a K\"ahler $2$-form $h$. In this case, $Y$ is also K\"ahler and since $Y$ is compact K\"ahler, it follows that all holomorphic $1$-forms are closed. Thus every infinitesimal deformation of $Y$ respects the Lagrangian condition. Furthermore, it follows from \cite{Voi} that in the Lagrangian case, all deformations are unobstructed. Hence there exists a local moduli space $B$ of complex Lagrangian submanifolds parametrising (sufficiently small) deformations of a given complex Lagrangian $Y_0 \subset \mathcal{M}$. A point $[Y] \in B$ is a complex Lagrangian $Y \subset \mathcal{M}$ which is deformation equivalent to $Y_0$. Moreover there is a natural isomorphism $T_{[Y]}B \cong H^0( Y , T^*_Y)$.\\

We recall from \cite{Hit2} how $B$ inherits a naturally defined special K\"ahler structure. Let $Z$ be a local universal family of deformations of the complex Lagrangian submanifold $Y_0 \subset \mathcal{M}$, so $Z$ is a complex manifold with a proper holomorphic surjective submersion $f : Z \to B$ and a holomorphic map $j : Z \to \mathcal{M}$ such that the restriction of $j$ to each fibre $L_b = f^{-1}(b)$ of $f$ gives an embedding $j : L_b \to \mathcal{M}$ whose image is the complex Lagrangian corresponding to the point $b \in B$. Let $\mathcal{V} = R^1 f_* \mathbb{R}$ be the vector bundle on $B$ whose fibre over $b \in B$ is given by the first cohomology $H^1( L_b , \mathbb{R})$ of the fibre $L_b$. The bundle $\mathcal{V}$ is equipped with a natural flat connection $\nabla$, the Gauss-Manin connection. The K\"ahler form $h$ on $\mathcal{M}$ yields a symplectic structure $\omega$ on $\mathcal{V}$ by setting
\[
\omega_b( \alpha , \beta ) = \int_{L_b} \alpha \wedge \beta \wedge h^{n-1},
\]
for all $\alpha , \beta \in \mathcal{V}_b = H^1( L_b , \mathbb{R})$. Clearly $\omega$ is preserved by the Gauss-Manin connection.

The natural isomorphism $T_b B \cong H^0( L_b , T^*_{L_b})$ together with the inclusion $H^0( L_b , T^*_{L_b}) \subset H^1( L_b , \mathbb{C})$ (recall $L_b$ is compact K\"ahler) yields a bundle map $\phi : TB \to \mathcal{V}_{\mathbb{C}}$:
\begin{equation*}
\xymatrix{
T_b B \ar@{}[r]|-*[@]{\cong} \ar@/_2pc/[rrr]_{\phi_b} & H^0(L_b , T^*_{L_b}) \ar@{}[r]|-*[@]{\subset} & H^1(L_b , \mathbb{C}) \ar@{}[r]|-*[@]{\cong} & \left( \mathcal{V}_{\mathbb{C}} \right)_b
}
\end{equation*}
More explicitly, let $X \in T_b B$ be a tangent vector at $b$. Let $\tilde{X}$ be a lift of $X$ to $L_b$, that is, $\tilde{X}$ is a section of $TZ|_{L_b}$ such that $f_*( \tilde{X}(y) ) = X$ for all $y \in L_b$. Then $\phi(X)$ is given by:
\begin{equation}\label{equ:phi}
\phi(X) = i_{\tilde{X}} j^*\Omega |_{L_b},
\end{equation}
which is a closed $1$-form on $L_b$ representing a class in $H^1(L_b , \mathbb{C})$. The map $\phi$ can be viewed as a $\mathcal{V}_{\mathbb{C}}$-valued $1$-form on $B$. We then have:
\begin{theorem}[Hitchin, \cite{Hit2}]\label{thm:clsk}
The $\mathcal{V}_{\mathbb{C}}$-valued $1$-form $\phi$ is $d_\nabla$-closed, so locally we can write $\phi = \nabla s$ for a section $s$ of $\mathcal{V}_{\mathbb{C}}$. Then $s$ satisfies the conditions (SK1)-(SK3) and hence locally defines a special K\"ahler structure on $B$. The special K\"ahler structure is independent of the choice of $s$, hence this construction gives rise to a globally defined special K\"ahler structure on $B$.
\end{theorem}
\begin{proof}
We give only a sketch of the proof here and refer the reader to \cite{Hit2} for further details. We have that $j^*\Omega$ is a closed $2$-form on $Z$. The complex Lagrangian condition means that $j^*\Omega |_{L_b} = 0$ for any fibre $L_b$. If we consider the Leray-Serre spectral sequence for $f : Z \to B$, we see that $j^*\Omega$ yields a class in $E_2^{1,1} = H^1( B , R^1 f_* \mathbb{C}) = H^1( B , \mathcal{V}_{\mathbb{C}})$, which is represented by the $\mathcal{V}_{\mathbb{C}}$-valued $1$-form $\phi$. This explains why $\phi$ is closed. Condition (SK1) follows from the isomorphism $T_b B \cong H^0( L_b , T^*_{L_b})$ and the inclusion $H^0(L_b , T^*_{L_b}) \subset H^1( L_b , \mathbb{C})$. Condition (SK2) follows from the fact that $H^0(L_b , T^*_{L_b}) \subset H^1( L_b , \mathbb{C})$ is Lagrangian with respect to $\omega$. Condition (SK3) follows from
\[
g_b( \alpha , \alpha) = Re\left( \frac{i}{2} \int_{L_b} \alpha \wedge \overline{\alpha} \wedge h^{n-1} \right),
\]
which is positive definite, since $h$ is a K\"ahler form.
\end{proof}

We give another way of understanding the $\mathcal{V}_{\mathbb{C}}$-valued $1$-form $\phi$. The fibre of the dual local system $\mathcal{V}^*_\mathbb{C}$ over $b \in B$ is given by the homology group $(\mathcal{V}_\mathbb{C}^*)_b = H_1( L_b , \mathbb{C})$. Let $\langle \; , \; \rangle : \mathcal{V}^*_\mathbb{C} \otimes \mathcal{V}_\mathbb{C} \to \mathbb{C}$ be the dual pairing:
\[
\langle \gamma , \alpha \rangle = \int_\gamma \alpha.
\]
Then $\phi$ is determined by
\[
\langle \gamma , \phi \rangle = \int_\gamma j^*\Omega,
\]
for all locally defined covariantly constant sections $\gamma$ of $\mathcal{V}^*_\mathbb{C}$, where the integral on the right hand side is fibrewise integration. More precisely, for any $b \in B$, choose an open neighborhood $U$ of $b$ over which $f : Z \to B$ is trivialisable: $f^{-1}(U) \cong L_b \times U$. Let $\hat{\gamma} \in H^{n-1}( L_b , \mathbb{C})$ be the Poincar\'e dual of $\gamma$. Then $\int_\gamma j^*\Omega$ is given by the integration over the fibres of $L_b \times U \to U$ of $j^*\Omega \wedge \hat{\gamma}$.

\begin{proposition}\label{prop:dmu}
Suppose that $\Omega = d\mu$, where $\mu$ is a holomorphic $1$-form on $\mathcal{M}$. Then $\phi = d_\nabla s$, where $s : B \to \mathcal{V}_\mathbb{C}$ is the section defined by
\[
s(b) = [ \mu|_{L_b} ] \in H^1( L_b , \mathbb{C}).
\]
\end{proposition}
\begin{proof}
First of all, note that the restriction $\mu|_{L_b}$ is a closed $1$-form because $d\mu = \Omega$ and $\Omega|_{L_b} = 0$, so the section $s$ is well-defined. Let $\gamma$ be a local constant section of $\mathcal{V}_\mathbb{C}$. Then
\begin{equation*}
\begin{aligned}
\langle \gamma , d_\nabla s \rangle &= d \langle \gamma , s \rangle - \langle d_\nabla \gamma , s \rangle \\
&= d \langle \gamma , s \rangle \quad \text{(as } \gamma \text{ is constant)} \\
&= d \left( \int_\gamma \mu \right) \\
&= \int_\gamma d\mu = \int_\gamma \Omega = \langle \gamma , \Omega \rangle,
\end{aligned}
\end{equation*}
where we have used the fact that exterior differentiation commutes with fibre integration. This shows $\phi = d_\nabla s$.
\end{proof}

Suppose we are in the situation of Proposition \ref{prop:dmu}. For any $b \in B$, choose a symplectic basis $a_1, \dots , a_n , b_1 , \dots , b_n$ of $H_1( L_b , \mathbb{C})$. If $U$ is any simply connected neighborhood of $b$, we can extend $a_i , b_i$ to be covariantly constant sections of $\mathcal{V}^*$ over $U$. Then the local holomorphic coordinates $z^i$ and $w_i$ of the special K\"ahler structure are given by
\[
z^i = \int_{a_i} \mu, \quad \quad w_i = \int_{b_i} \mu.
\]

\begin{example}[Holomorphic symplectic surfaces]
Let $\mathcal{M}$ be a holomorphic symplectic surface, i.e. a complex surface with trivial canonical bundle. Let $\Omega$ be the symplectic form on $\mathcal{M}$. Then any complex $1$-dimensional submanifold $S \subset \mathcal{M}$ is automatically Lagrangian, because there are no $(2,0)$-forms on $S$. The moduli space $B$ of deformations of $S$ then carries a natural special K\"ahler structure. We note that in this case since $n=1$, it is not necessary to assume the existence of a K\"ahler form $h$ on $\mathcal{M}$. Indeed, the symplectic structure on $\mathcal{V}_{\mathbb{C}}$ is given by
\[
\omega_b( \alpha , \beta ) = \int_{L_b} \alpha \wedge \beta,
\]
which is defined without assuming the existence of $h$. Moreover, since $S$ is a Riemann surface it is automatically K\"ahler. As we will see, this example is closely related to the Hitchin system, where $\mathcal{M} = T^*\Sigma$ is the cotangent bundle of a Riemann surface $\Sigma$ and $S \subset T^*\Sigma$ is a spectral curve.
\end{example}

\begin{example}[Complex integrable systems]\label{ex:intsys}

Suppose that $\mathcal{M}$ admits a proper holomorphic Lagrangian fibration $f : \mathcal{M} \to B$, by which we mean $B$ is a complex $n$-manifold and $f : \mathcal{M} \to B$ is a proper, surjective, holomorphic map whose fibres are Lagrangian with respect to $\Omega$. Assume further that the fibres of $f$ are connected. Then Liouville's theorem implies that the fibres of $f$ are in fact complex tori. We shall refer to the data $(\mathcal{M} , \Omega , B , f)$ as a complex integrable system. Since the normal bundle to the fibres is given by $f^*(TB)$, it follows that the deformations of any given fibre $L_b \subset \mathcal{M}$ are precisely the other fibres of $f$. Hence $B$ is the moduli space of deformations of any given fibre. If we make the additional assumption that $\mathcal{M}$ admits a K\"ahler $2$-form $h$, then $B$ inherits a natural special K\"ahler structure (which depends on the choice of $h$). 
\end{example}

\section{Higgs bundles and the Hitchin system}\label{sec:higgsb}

\subsection{Review of Higgs bundles}

Let $\Sigma$ be a compact Riemann surface of genus $g > 1$ and let $K$ denote the canonical bundle of $\Sigma$.

\begin{definition}
A Higgs bundle on $\Sigma$ of rank $n$ and degree $d$ is a pair $(E , \Phi)$, where $E$ is a holomorphic vector bundle on $\Sigma$ of rank $n$, degree $d$ and $\Phi$ is a holomorphic bundle map $\Phi : E \to E \otimes K$, called the Higgs field.
\end{definition}

Recall that the slope $\mu(E)$ of a holomorphic vector bundle $E$ on $\Sigma$ is defined to be the number $\mu(E) = {\rm deg}(E)/{\rm rank}(E)$.

\begin{definition}
A Higgs bundle $(E , \Phi)$ is called semistable if for all proper, non-zero subbundles $F \subset E$ such that $\Phi(F) \subseteq F \otimes K$, we have $\mu(F) \le \mu(E)$.
\end{definition}

Using geometric invariant theory \cite{nit}, one constructs a moduli space $\mathcal{M}_{n,d}$ of rank $n$, degree $d$ semistable Higgs bundles up to a suitable notion of equivalence (called $S$-equivalence). Here we will recall only the details of $\mathcal{M}_{n,d}$ which are relevant to understanding the special K\"ahler geometry of the Hitchin system. The moduli space $\mathcal{M}_{n,d}$ is a quasi-projective, complex algebraic variety of dimension $2n^2(g-1)+2$. The moduli space is in general singular, but the smooth locus $\mathcal{M}_{n,d}^{{\rm sm}}$ has a naturally defined holomorphic symplectic $2$-form $\Omega$. The smooth locus $\mathcal{M}_{n,d}^{{\rm sm}}$ is a hyperk\"ahler manifold with a triple of complex structures $I,J,K$ and corresponding K\"ahler $2$-forms $\omega_I , \omega_J , \omega_K$. The complex structure which arises from viewing $\mathcal{M}_{n,d}$ as the moduli space of Higgs bundles is customarily taken to be $I$ and it is the only complex structure of relevance to us here. So we will regard $\mathcal{M}_{n,d}^{\rm sm}$ as a K\"ahler manifold $(\mathcal{M}^{\rm sm}_{n,d} , I , \omega_I)$ equipped with a holomorphic symplectic form $\Omega = \omega_J + i\omega_K$. One can also consider moduli spaces of Higgs bundles $(E,\Phi)$, where $\Phi$ is trace-free and $E$ has fixed determinant. This gives a subvariety of $\mathcal{M}_{n,d}$. As the corresponding special K\"ahler geometry is obtained by restriction, it is sufficient for our purposes to consider just the case of $\mathcal{M}_{n,d}$.

Associated to $\mathcal{M}_{n,d}$ is a complex integrable system, known as the {\em Hitchin system} \cite{hit3}, which is defined as follows. If $A$ is a complex $n \times n$ matrix, write the characteristic polynomial of $A$ as
\[
{\rm det}( \lambda - A ) = \lambda^n + p_1(A) \lambda^{n-1} + \dots + p_n(A).
\]
The coefficients $p_1 , \dots , p_n$ of the characteristic polynomial can be viewed as maps $\mathfrak{gl}(n,\mathbb{C}) \to \mathbb{C}$ which are are well known to be a basis for the ring of conjugation-invariant polynomial functions on $\mathfrak{gl}(n,\mathbb{C})$. The Hitchin system is the analogue of this where the matrix $A$ is replaced by a Higgs field. If $(E,\Phi)$ is a rank $n$ Higgs bundle, then by conjugation invariance, $p_j(\Phi)$ is a well-defined section of $K^j$. Define
\[
B = \bigoplus_{j=1}^n H^0( \Sigma , K^j ).
\]
Then we have a natural map $f : \mathcal{M}_{n,d} \to B$, called the {\em Hitchin map}, which sends a Higgs bundle $(E,\Phi)$ to $(p_1(\Phi) , p_2(\Phi) , \dots , p_n(\Phi))$. It is known \cite{hit1,sim} that $f$ is a proper,  surjective, holomorphic map whose non-singular fibres are Lagrangian submanifolds with respect to $\Omega$. Let $B^{\rm reg} \subset B$ denote the {\em regular locus}, i.e. the locus of points $b \in B$ over which $f$ is a submersion. Let $\mathcal{M}_{n,d}^{\rm reg} \subset \mathcal{M}_{n,d}$ denote the locus of points in $\mathcal{M}_{n,d}$ lying over $B^{\rm reg}$, so that $f : \mathcal{M}_{n,d}^{\rm reg} \to B^{\rm reg}$ is a proper, holomorphic surjective submersion of complex manifolds (since $B$ is smooth and $f : \mathcal{M}_{n,d}^{\rm reg} \to B^{\rm reg}$ is a submersion, it follows that $\mathcal{M}_{n,d}^{\rm reg}$ is contained in the smooth locus $\mathcal{M}_{n,d}^{\rm sm}$). As seen in Theorem \ref{thm:specdat} below, the fibres of $f$ over $B^{\rm reg}$ are connected and hence, as in Example \ref{ex:intsys}, the fibres are complex tori. We then have that $B^{\rm reg}$ can be identified with the moduli space of deformations of any given fibre in $\mathcal{M}_{n,d}^{\rm reg}$ and thus $B^{\rm reg}$ carries a natural special K\"ahler geometry. The complement $\mathcal{D} = B \setminus B^{\rm reg}$, called the {\em discriminant locus}, is the locus of all singular fibres of $f : \mathcal{M}_{n,d} \to B$. It is known that $\mathcal{D}$ is an irreducible hypersurface of $B$ \cite{kp}. 

\subsection{Spectral curves}
The fibres of the Hitchin system can be described using the notion of {\em spectral curves} \cite{bnr,hit3}. Let $\pi : T^*\Sigma \to \Sigma$ be the projection from $T^*\Sigma$ to $\Sigma$. Observe that $T^*\Sigma$ is the total space of the canonical bundle $K \to \Sigma$. Therefore the pullback $\pi^*(K)$ has a natural section $\lambda$, the {\em tautological section}, defined by the property that if $p \in K_x$, then
\[
\lambda(p) = p \in \left( \pi^*(K) \right)_p = K_x.
\]
Let $b = (p_1 , p_2 , \dots , p_n ) \in B$, so $p_j \in H^0( \Sigma , K^j)$. To the point $b$, we associate a section $p = p_b \in H^0( T^*\Sigma , \pi^*(K^n) )$, given by
\[
p_b = \lambda^n + \pi^*(p_1) \lambda^{n-1} + \dots + \pi^*(p_n).
\]
The spectral curve $S_b$ associated to $b \in B$ is defined as the zero divisor of $p_b$ in $T^*\Sigma$. Thus $S_b$ is defined as a subscheme of $T^*\Sigma$. Bertini's theorem implies that for generic points $b \in B$, the spectral curve $S_b \subset T^*\Sigma$ is smooth. In fact, it can be shown that $S_b$ is smooth if and only if $b \in B^{\rm reg}$ \cite{kp}. Thus, the discriminant locus $\mathcal{D}$ is also the locus of singular spectral curves. 

If $S_b$ is a spectral curve, we denote by $\pi : S_b \to \Sigma$ the restriction of $\pi$ to $S_b$. Let $\tilde{d} = d + (g-1)n(n-1)$. For $b \in B^{\rm reg}$, let $Jac_{\tilde{d}}(S_b)$ denote the degree $\tilde{d}$ component of the Picard variety of $S_b$, which is a torsor for the Jacobian of $S_b$. If $L \in Jac_{\tilde{d}}(S_b)$, then by Grothendieck-Riemann-Roch, one sees that the push-forward $E = \pi_*L$ is a rank $n$, degree $d$ vector bundle on $\Sigma$. The tautological section $\lambda$, viewed as a map $\lambda : L \to L \otimes \pi^*K$ pushes down to a map $\Phi : E \to E \otimes K$ and hence the pair $(E , \Phi)$ is a Higgs bundle. Since $p_b = 0$ on $S_b$, it follows that:
\[
\Phi^n + \pi^*(p_1) \Phi^{n-1} + \dots + \pi^*(p_n) = 0.
\]
Since $b \in B^{\rm reg}$, $S_b$ is smooth and it follows that $p_b$ is irreducible. Thus $p_b$ must be the characteristic polynomial of $\Phi$, in other words $(E,\Phi)$ belongs to the fibre of $\mathcal{M}_{n,d}$ over $b \in B$. In this way, we obtain all Higgs bundles in the fibre over $b$:

\begin{theorem}[\cite{hit3,bnr}]\label{thm:specdat}
Let $b \in B^{\rm reg}$. The map sending $L \in Jac_{\tilde{d}}(S_b)$ to $(E , \Phi)$ described above gives an isomorphism between $Jac_{\tilde{d}}(S_b)$ and the fibre of $\mathcal{M}_{n,d}$ over $b \in B$.
\end{theorem}

\subsection{On deformations of spectral curves}

Let $\theta \in \Omega^{1,0}(T^*\Sigma)$ denote the canonical $1$-form on $T^*\Sigma$. Then $d\theta$ is the canonical holomorphic symplectic $2$-form on $T^*\Sigma$ and gives a trivialisation of the canonical bundle of $T^*\Sigma$. Now let $S \subset T^*\Sigma$ be a non-singular spectral curve. We denote by $K_S$ the canonical bundle of $S$ and $N_S$ the normal bundle. By definition of $N_S$, we have a short exact sequence
\begin{equation*}
0 \to K_S^{-1} \to T_{(T^*\Sigma)}|_S \to N_S \to 0.
\end{equation*}
Taking determinants and using $d\theta$, we get an isomorphism $N_S \cong K_S$. Explicitly, the isomorphism is given by
\begin{equation*}
\varphi : N_S \to K_S, \quad \varphi(V) = i_V d\theta |_S
\end{equation*}
which sends a normal vector field $V$ to the contraction of $d\theta$ with $V$. Next, since $S$ is by definition a divisor of the linear system of sections of $\pi^*(K^n)$, we have by the adjunction formula that $N_S \cong \pi^*(K^n)|_S$. We will make this isomorphism explicit as follows. Suppose that $S$ is the zero divisor of $p \in H^0( T^*\Sigma , \pi^*(K^n))$. Choose an open covering $\{ U_i \}$ of $T^*\Sigma$ over which $\pi^*(K^n)$ is trivial. Let $g_{ij} : U_i \cap U_j \to \mathbb{C}^*$ be the transition functions, so $p$ corresponds to a collection of functions $s_i : U_i \to \mathbb{C}$ such that $s_i = g_{ij}s_j$ on $U_i \cap U_j$ and $s_i |_S = 0$ for all $i$. Now let $V$ be a normal vector field along $S$. Denote by $\partial_V s_i$ the derivative of $s_i$ in the direction $V$. Then $\partial_V(s_i) |_S = g_{ij} \partial_V( s_j) |_S$, because $s_j|_S = 0$. Therefore, $\{ \partial_V(s_i)|_S \}_i$ is a well-defined section of $\pi^*(K^n)|_S$, which we denote simply as $\partial_V p |_S$. The desired isomorphism is
\[
\varphi' : N_S \to \pi^*(K^n), \quad \varphi'(V) = \partial_V p|_S.
\]
\begin{lemma}
Let $b \in B^{\rm reg}$. The map
\[
\rho : B = \bigoplus_{j=1}^n H^0( \Sigma , K^j) \to H^0( S_b , \pi^*(K^n)) 
\]
given by
\[
\rho( b_1 , b_2 , \dots , b_n) = \pi^*(b_1)\lambda^{n-1} + \pi^*(b_2)\lambda^{n-2} + \dots + \pi^*(b_n)
\]
is an isomorphism.
\end{lemma}
\begin{proof}
From \cite{bnr}, we have
\[
\pi_* \mathcal{O}_S = \mathcal{O}_\Sigma \oplus K^{-1} \oplus \dots \oplus K^{-(n-1)}.
\]
Multiplying both sides by $K^n$ and taking global sections gives the result.
\end{proof}

Let $b \in B^{\rm reg}$ and $p_b \in H^0( T^*\Sigma , \pi^*(K^n))$ the corresponding section of $\pi^*(K^n)$. A tangent vector $X \in T_b B^{\rm reg} \cong B$ gives rise to a deformation of $p_b$ and hence to a deformation of the divisor $S_b \subset T^*\Sigma$ of $p_b$. Such a deformation is described by a section of $N_{S_b}$, hence we get a map, known as the {\em characteristic map} (cf., \cite{gri}):
\begin{equation}\label{equ:charmap}
\chi : T_b B^{\rm reg} \cong B = \bigoplus_{j=1}^n H^0(\Sigma , K^j) \to H^0( S_b , N_{S_b}).
\end{equation}

\begin{proposition}\label{prop:deform}
The characteristic map is given by $\chi = -(\varphi')^{-1} \circ \rho$.
\end{proposition}
\begin{proof}
Let $b(t)$ be a $1$-parameter family in $B^{\rm reg}$ with $b(0) = b$, and $b'(0) \in T_b B^{\rm reg} \cong B$ a tangent vector at $t=0$. The derivative of $p_{b(t)}$ at $t=0$ is clearly given by $\rho( b'(0))$. Let $x(t)$ denote a $1$-parameter family of points in $T^*\Sigma$ such that $x(t)$ lies on $S_{b(t)}$ for all times $t$, i.e. $p_{b(t)}( x(t) ) = 0$. Note that the projection of $x'(0)$ to the normal bundle is given by $V(x)$, where $V = \chi(b'(0)) \in H^0( S_b , N_{S_b})$ is the normal vector field describing the deformation of $S_b$ in $T^*\Sigma$. Expanding $p_{b(t)}(x(t)) = 0$ to first order at $t=0$, we get $\rho(b'(0))(x) + \partial_{V(x)} p_b = 0$, or $\varphi'( V(x)) = -\rho(b'(0))(x)$. So $V = \chi(b'(0)) = -(\varphi')^{-1} \rho( b'(0))$, as required.
\end{proof}

Note that $-(\varphi')^{-1} \circ \rho : T_b B^{\rm reg} \to H^0( S_b , N_{S_b})$ is an isomorphism. Thus as a consequence of Proposition \ref{prop:deform}, we see that $B^{\rm reg}$ can also be identified with the moduli space of deformations of $S_b \subset $ in $T^*\Sigma$.

\section{Special K\"ahler geometry of the Hitchin system}\label{sec:skhs}

\subsection{Two special K\"ahler geometries}

We now have two possible ways of obtaining a special K\"ahler geometry on $B^{\rm reg}$:
\begin{itemize}
\item[(1)]{View $B^{\rm reg}$ as parametrising a family of spectral curves $S \subset T^*\Sigma$, which are complex Lagrangians, or}
\item[(2)]{View $B^{\rm reg}$ as parametrising a family of complex Lagrangians in $\mathcal{M}_{n,d}^{\rm reg}$, the fibres of the Hitchin system.}
\end{itemize}

We will show that both of these give rise to the same special K\"ahler geometry on $B^{\rm reg}$. 

\begin{definition}
Let $(\mathcal{V} , \omega , \nabla)$ and $(\mathcal{V}' , \omega' , \nabla')$ denote the flat symplectic bundles on $B^{\rm reg}$ corresponding to (1) and (2) above. Let $\phi : TB^{\rm reg} \to \mathcal{V}_\mathbb{C}$ and $\phi' : TB^{\rm reg} \to \mathcal{V}_\mathbb{C}'$) denote the $\mathcal{V}_\mathbb{C}$ and $\mathcal{V}'_\mathbb{C}$-valued $1$-forms corresponding to (1) and (2).
\end{definition}

Consider first the special K\"ahler geometry on $B^{\rm reg}$ given by $( \mathcal{V} , \omega , \nabla , \phi)$. Recall that $\theta$ is the canonical $1$-form on $T^*\Sigma$ and $d\theta$ is the symplectic form on $T^*\Sigma$. The map $B^{\rm reg} \ni b \mapsto [ \theta|_{S_b}] \in H^1( S_b , \mathbb{C})$ defines a section of $\mathcal{V}_\mathbb{C}$ which by abuse of notation we will denote by $\theta$. Then Proposition \ref{prop:dmu} gives $\phi = d_\nabla \theta$, and hence the section $\theta$ determines a special K\"ahler geometry on $B^{\rm reg}$.

Next consider the special K\"ahler geometry on $B^{\rm reg}$ by $( \mathcal{V}' , \omega' , \nabla' , \phi')$. For this we need to introduce the canonical $1$-form on $\mathcal{M}_{n,d}$. Let $(E , \Phi)$ be a semistable Higgs bundle in the non-singular locus of $\mathcal{M}_{n,d}$. The tangent space to $(E,\Phi)$ is given by the hypercohomology group $\mathbb{H}^1( \Sigma , End(E) \buildrel ad_\Phi \over \longrightarrow End(E) \otimes K )$. Thus tangent vectors to $(E,\Phi)$ are represented by pairs $(\dot{A} , \dot{\Phi}) \in \Omega^{0,1}(\Sigma , End(E)) \oplus \Omega^{0,0}(\Sigma , End(E) \otimes K)$ satisfying $\overline{\partial}_E \dot{\Phi} + [ \dot{A} , \Phi] = 0$. Here $\dot{A}$ represents a deformation of holomorphic structure of $E$ and $\dot{\Phi}$ represents a deformation of the Higgs field $\Phi$. The natural map $\mathbb{H}^1( \Sigma , End(E) \buildrel ad_\Phi \over \longrightarrow End(E) \otimes K ) \to H^1(\Sigma , End(E))$, sending a deformation to $(E,\Phi)$ to a deformation of $E$ alone is given in terms of Dolbeault representatives by $(\dot{A} , \dot{\Phi}) \mapsto \dot{A}$. 

The holomorphic symplectic form $\Omega$ on $\mathcal{M}_{n,d}$ is of the form $\Omega = d\mu$, where $\mu$ is a holomorphic $(1,0)$-form, the {\em canonical $1$-form}. Up to an overall scale factor, which is not important for us, the canonical $1$-form is given by
\[
\mu_{(E,\Phi)}( \dot{A} , \dot{\Phi}) = \int_{\Sigma} Tr( \Phi \dot{A} ).
\]

By abuse of notation, let $\mu$ denote the section of $\mathcal{V}'_\mathbb{C}$ on $B^{\rm reg}$ given by $B^{\rm reg} \ni b \mapsto [\mu|_{L_b}] \in H^1( L_b , \mathbb{C})$, where $L_b = f^{-1}(b) \cong Jac_{\tilde{d}}(S_b)$ is the fibre of the Hitchin map over $b$. Then by Proposition \ref{prop:dmu}, $\phi' = d_{\nabla'} \mu$. Hence the section $\mu$ defines a special K\"ahler geometry on $B^{\rm reg}$.

\begin{proposition}
There is a natural isomorphism of local systems $u : (\mathcal{V} , \nabla) \to (\mathcal{V}' , \nabla')$. Under this isomorphism, $\omega$ and $\omega'$ agree up to a positive constant factor.
\end{proposition}
\begin{proof}
Let $Z^{\rm reg}$ denote the universal moduli space of spectral curves, which may be defined as
\[
Z^{\rm reg} = \{ (x , b ) \in T^*\Sigma \times B^{\rm reg} \; | \; p_b(x) = 0 \}.
\]
Thus $Z^{\rm reg}$ is a fibre bundle $q : Z^{\rm reg} \to B^{\rm reg}$ over $B^{\rm reg}$ whose fibre over $b \in B^{\rm reg}$ is the spectral curve $S_b$. Let $j : Z^{\rm reg} \to T^*\Sigma$ be given by $j(x,b) = x$. Then $\mathcal{V} = R^1 q_*\mathbb{R}$ is the local system $\mathcal{V}_b = H^1( S_b , \mathbb{R})$ equipped with the Gauss-Manin connection and $\omega$ is given by 
\[
\omega_b(\alpha , \beta) = \int_{S_b} \alpha \wedge \beta.
\]
Recall the Hitchin map $f : \mathcal{M}_{n,d}^{\rm reg} \to B^{\rm reg}$ whose fibre over $b \in B^{\rm reg}$ is $Jac_{\tilde{d}}(S_b)$. Then $\mathcal{V}' = R^1 f_* \mathbb{R}$ is the local system $\mathcal{V}'_b = H^1( Jac_{\tilde{d}}(S_b) , \mathbb{R})$ with the Gauss-Manin connection and $\omega'$ given by
\[
\omega'_b(\alpha' , \beta') = \int_{Jac_{\tilde{d}}(S_b)} \alpha' \wedge \beta' \wedge h^{n-1}.
\]
Recall that $Jac_{\tilde{d}}(S_b)$ is a torsor over $Jac(S_b)$. Hence there is a canonical isomorphism $H^1( Jac_{\tilde{d}}(S_b) , \mathbb{R}) \cong H^1(Jac(S_b) , \mathbb{R})$. We may identify the Jacobian $Jac(S_b)$ with $H_1( S_b , \mathbb{R})/H_1(S_b , \mathbb{Z})$, hence we have a canonical isomorphism $H^1(Jac(S_b) , \mathbb{R}) \cong H^1(S_b , \mathbb{R})$. By composing we get a canonical isomorphism $H^1( S_b , \mathbb{R}) \cong H^1( Jac_{\tilde{d}}(S_b) , \mathbb{R})$. Clearly this isomorphism can be carried out fibrewise over $B^{\rm reg}$ to give an isomorphism of flat vector bundles $\mathcal{V} \cong \mathcal{V}'$. Furthermore, the natural K\"ahler form $h = \omega_I$ on $\mathcal{M}_{n,d}$ restricted to the fibre $Jac_{\tilde{d}}(S_b)$ gives a multiple of the usual principal polarisation on $Jac_{\tilde{d}}(S_b)$. It follows that $\omega$ and $\omega'$ agree up to a positive constant factor (which is independent of $b$).
\end{proof}

\begin{proposition}\label{prop:canonicals}
Let $b \in B^{\rm reg}$. Under the canonical isomorphism $H^1(Jac_{\tilde{d}}(S_b) , \mathbb{C}) \cong H^1( S_b , \mathbb{C})$, we have $\mu|_{Jac_{\tilde{d}}(S_b)} \cong \theta|_{S_b}$.
\end{proposition}
\begin{proof}
The restriction of $\mu$ to the fibre $Jac_{\tilde{d}}(S_b)$ over $b$ is a holomorphic $1$-form, which is necessarily translation invariant, since this is true of all holomorphic $1$-forms on a complex torus. The tangent space to $Jac_{\tilde{d}}(S_b)$ at a point $[L] \in Jac_{\tilde{d}}(S_b)$ is canonically isomorphic to $H^1(S_b , \mathcal{O})$. Under the isomorphism $H^1(Jac_{\tilde{d}}(S_b) , \mathbb{C}) \cong H^1( S_b , \mathbb{C})$, the pairing of a tangent vector with a holomorphic $1$-form coincides with the Serre duality pairing $H^0(S_b , K_{S_b}) \otimes H^1(S_b , \mathcal{O}) \to \mathbb{C}$. In other words, let $X \in T_{[L]} Jac_{\tilde{d}}(S_b)$ and let $\alpha_X$ be the corresponding element of $H^1(S_b , \mathcal{O})$. Then the statement of the proposition is equivalent to showing:
\[
\mu(X) = \int_{S_b} \theta \wedge \alpha_X.
\]
The point $[L] \in Jac_{\tilde{d}}(S_b)$ corresponds to a line bundle $L \to S_b$. Under the spectral data construction, the Higgs bundle $(E,\Phi)$ corresponding to $L$ is given by $E = \pi_*(L)$ and $\Phi$ by pushing forward the map $\lambda : L \to L \otimes \pi^*K$. We view the holomorphic line bunde $L$ as a $\mathcal{C}^\infty$-line bundle together with a $\overline{\partial}$-operator $\overline{\partial}_L$. Then the tangent vector $X$ corresponds to the tangent at $t=0$ of the $1$-parameter family of deformations of $L$ given by $\overline{\partial}_{L_t} = \overline{\partial}_L + t \alpha_X$. Similarly view $E$ as a $\mathcal{C}^\infty$-vector bundle with a $\overline{\partial}$-operator $\overline{\partial}_E$. Let $E_t = \pi_*(L_t)$. We will construct an explicit family of $\overline{\partial}$-operators $\overline{\partial}_{E_t}$ on the fixed $\mathcal{C}^\infty$-vector bundle $E$ such that $(E , \overline{\partial}_{E_t} )$ is isomorphic to $E_t$.

By the Dolbeault Lemma, adding a $\overline{\partial}$-exact term to $\alpha_X$ if necessary, we can assume that $\alpha_X$ vanishes identically in a neighborhood of each point of $S_b$ lying over a branch point. Now let $U \subset \Sigma$ be an open, simply-connected subset containing no branch points. Then the pre-image $\pi^{-1}(U) = U_1 \cup U_2 \cup \dots \cup U_n$ is the disjoint union of $n$ open subsets of $S_b$, and the restriction $\pi : U_j \to U$ of $\pi$ to each $U_j$ is a diffeomorphism. Over $U$, we have a canonical isomorphism $E|_U \cong L|_{U_1} \oplus L|_{U_2} \oplus \dots \oplus L|_{U_n}$. Define an $End(E)$-valued $(0,1)$-form $\dot{A}|_U$ on $U$ by
\[
\dot{A}|_U = \text{diag}\left( (\pi^*)^{-1} (\alpha_X|_{U_1}) , \dots , (\pi^*)^{-1} (\alpha_X|_{U_n}) \right).
\]
Here $\pi^* : \Omega^{0,1}(U) \to \Omega^{0,1}(U_j)$ denotes the pullback of $(0,1)$-forms and $(\pi^*)^{-1} : \Omega^{0,1}(U_j) \to \Omega^{0,1}(U)$ the inverse map. The $\dot{A}|_U$ defined in this way patch together to give an $End(E)$-valued $(0,1)$-form on $\Sigma$ minus the branch points. But since $\alpha_X$ vanishes in a neighborhood of each point of $S_b$ lying over a branch point, we have that $\dot{A}$ vanishes in a punctured neighborhood of each branch point. We extend $\dot{A}$ by zero over each branch point to get a well-defined $\dot{A} \in \Omega^{0,1}( \Sigma , End(E))$. Set $\overline{\partial}_{E_t} = \overline{\partial}_E + t\dot{A}$. By construction of $\dot{A}$, it is clear that $(E , \overline{\partial}_{E_t})$ is isomorphic to $E_t$. Observe that $\lambda : L_t \to L_t \otimes \pi^* K$ pushes down to $\Phi$ independent of $t$ and note that $\Phi$ is holomorphic with respect to $\overline{\partial}_{E_t}$ for all $t$ (because $\dot{A}$ and $\Phi$ are simultaneuosly diagonalisable away from the branch points). So $(E_t , \Phi )$ is the $1$-parameter family of Higgs bundles corresponding to $L_t$. In particular, differentiating at $t=0$, we have that $(\dot{A} , \dot{\Phi}) = (\dot{A} , 0) \in \mathbb{H}^1( \Sigma , End(E) \buildrel ad_\Phi \over \longrightarrow End(E) \otimes K )$ is the tangent vector corresponding to $X \in T_{[L]} Jac_{\tilde{d}}(S_b)$. Let $U \subset \Sigma$ be a simply-connected open subset of $\Sigma$ containing no branch points, as above, and let $\psi$ be a smooth compactly supported function on $U$. Over $U$, we have
\[
\Phi = \text{diag}\left( \lambda|_{U_1} \circ \pi^{-1} , \dots , \lambda|_{U_n} \circ \pi^{-1} \right)
\]
and so
\begin{equation*}
\begin{aligned}
\int_U \psi Tr( \Phi \dot{A} ) &= \sum_{j=1}^n \int_U \psi \left( \lambda|_{U_j} \circ \pi^{-1} \right) \wedge \left( (\pi^*)^{-1}( \alpha_X|_{U_j}) \right)\\
&= \sum_{j=1}^n \int_{U_j} \pi^*(\psi) \; \theta \wedge \alpha_X \\
&= \int_{\pi^{-1}(U)} \pi^*(\psi) \; \theta \wedge \alpha_X.
\end{aligned}
\end{equation*}
Combining this with a partition of unity argument, we get
\begin{equation*}
\begin{aligned}
\mu(X) = \int_\Sigma Tr( \Phi \dot{A}) = \int_S \theta \wedge \alpha_X
\end{aligned}
\end{equation*}
as required.
\end{proof}

\begin{corollary}
The special K\"ahler geometries on $B^{\rm reg}$ given by (1) and (2) coincide (up to a constant rescaling of the metric $g$ and symplectic form $\omega$).
\end{corollary}

\subsection{Special K\"ahler geometry of $B^{\rm reg}$}
We summarise our findings so far concerning the special K\"ahler geometry of the Hitchin system and make some further observations. Recall that $B^{\rm reg}$ is the regular locus of the Hitchin base, that is $B^{\rm reg} = B \setminus \mathcal{D}$, where $\mathcal{D}$ is the locus of singular spectral curves. On $B^{\rm reg}$ we have the flat symplectic bundle $(\mathcal{V} , \omega, \nabla )$ whose fibre over $b$ is $\mathcal{V}_b = H^1( S_b , \mathbb{R})$. The flat connection $\nabla$ is the Gauss-Manin connection and the symplectic structure $\omega$ is the intersection form $\omega(\alpha , \beta) = \int_{S_b} \alpha \wedge \beta$. Let $\theta$ denote the canonical $1$-form on $T^*\Sigma$. We think of $\theta$ as a section $\theta : B^{\rm reg} \to \mathcal{V}_{\mathbb{C}}$ which sends $b$ to $[ \theta |_{S_b}] \in H^1(S_b , \mathbb{C})$. For any $b \in B^{\rm reg}$, let $a_1, \dots , a_{g_S} , b_1 , \dots , b_{g_S}$ be a symplectic basis for $H_1(S_b , \mathbb{R})$ (where $g_S$ denotes the genus of the spectral curves). We can extend $a_1 , \dots , b_{g_S}$ to covariantly constant sections of $\mathcal{V}^*$ in any simply connected open neighborhood $U \subset B^{\rm reg}$ of $b$. Then the holomorphic coordinate systems $(z^1 , \dots , z^{g_S})$ and $(w_1 , \dots , w_{g_S})$ are given by:
\begin{equation}\label{equ:hitcoords}
z^i(u) = \int_{a_i} \theta, \quad \quad w_i(u) = \int_{b_i} \theta.
\end{equation}
For each $u \in U$, let $\omega_1(u) , \dots , \omega_{g_S}(u)$ be the corresponding normalised basis of holomorphic $1$-forms on $S_u$, which are characterised by:
\begin{equation}\label{equ:basisholo}
\int_{a_i} \omega_j(u) = \delta_{ij}.
\end{equation}
The period matrix of $S_u$ with respect to the symplectic basis $a_1 , \dots , b_{g_S}$ is
\[
\tau_{ij}(u) = \int_{b_i} \omega_j(u).
\]
In these coordinates the special K\"ahler metric is given in terms of its K\"ahler form by
\[
\omega = -\frac{i}{2} Im( \tau_{ij} ) dz^i \wedge d\overline{z}^j.
\]
Since $\theta$ is a holomorphic $1$-form, it can be written as a linear combination of the $\omega_i$. From Equations (\ref{equ:hitcoords}) and (\ref{equ:basisholo}), we immediately get
\[
\theta = z^i \omega_i.
\]
This equation holds not just at the level of cohomology classes, but as $1$-forms. Combining this with (\ref{equ:hitcoords}), we also find
\[
w_i = \int_{b_i} \theta = \int_{b_i} z^j \omega_j = z^j \tau_{ij}.
\]
Such a relation between the $z$ and $w$-coordinates does not hold for special K\"ahler manifolds in general. We now deduce a simple formula for the K\"ahler potential:
\begin{proposition}
The K\"ahler potential $K$ in Equation (\ref{equ:kahlerpot}) is given by:
\[
K = -\frac{i}{4} \int_{S} \theta \wedge \overline{\theta}.
\]
\end{proposition}
\begin{proof}
This is a straightforward computation:
\begin{equation*}
\begin{aligned}
-\frac{i}{4} \int_S \theta \wedge \overline{\theta} &= \frac{1}{4i} \sum_{i=1}^{g_S} \left( \int_{a_i} \theta \int_{b_i} \overline{\theta} - \int_{b_i} \theta \int_{a_i} \overline{\theta} \right) \\
&= \frac{1}{4i} \sum_{i=1}^{g_S} \left( z^i \overline{w_i} - w_i \overline{z^i} \right) \\
&= \frac{1}{2} \sum_{i=1}^{g_S} Im( z^i \overline{w_i} ) \\
&= -\frac{1}{2} \sum_{i=1}^{g_S} Im( w_i \overline{z^i} ) = K.
\end{aligned}
\end{equation*}
\end{proof}

\begin{remark}
An interesting feature of this formula is that is does not depend on a choice of symplectic basis and $K$ is thus a {\em globally defined} K\"ahler potential on $B^{\rm reg}$.
\end{remark}

Recall that the moduli space of Higgs bundles has a natural $\mathbb{C}^*$-action given by rescaling the Higgs field $(E , \Phi) \mapsto (E , c \Phi)$, $c \in \mathbb{C}^*$. The $\mathbb{C}^*$-action on $\mathcal{M}_{n,d}$ is compatible with a $\mathbb{C}^*$-action on the base given by
\[
c (  a_1 , a_2 , \dots , a_n ) = ( ca_1 , c^2 a_2 , \dots , c^n a_n).
\]
Let $\xi$ be the vector field on $B$ generating this action.
\begin{proposition}
In terms of local special K\"ahler coordinates $(z^1 , \dots , z^{g_S})$ or $(w_1 , \dots , w_{g_S})$ on $B^{\rm reg}$, we have:
\[
\xi = z^i \frac{\partial}{\partial z^i} = w_i \frac{\partial}{\partial w_i}.
\]
\end{proposition}
\begin{proof}
In terms of spectral data, the $\mathbb{C}^*$-action on Higgs bundles corresponds to the natural $\mathbb{C}^*$ scaling action on the fibres of $T^*\Sigma \to \Sigma$. It follows that the $\mathbb{C}^*$-action scales $\theta$ linearly. On the other hand, for any $c \in \mathbb{C}^*$, the action of $c$ on $T^*\Sigma$ induces an isomorphism of spectral curves $c : S_b \to S_{cb}$. Supose that $U$ is an open simply connected subset of $B^{\rm reg}$ containing $b$. Then for all $c \in \mathbb{C}^*$ sufficiently close to $1$, we have $cb \in U$ and the isomorphism $H_1( S_b , \mathbb{R}) \cong H_1( S_{cb} , \mathbb{R})$ induced by multiplication by $c$ clearly agrees with parallel translation by the Gauss-Manin connection. In particular, since $z^i = \int_{a_i} \theta$, each of the coordinate functions $z^i$ must scale linearly with the $\mathbb{C}^*$-action and hence $\xi(z^i) = z^i$ for each $i$. This gives $\xi = z^i \frac{\partial}{\partial z^i}$ as required. The same argument applied to $B$-cycles gives $\xi(w_i) = w_i$, hence $\xi = w_i \frac{\partial}{\partial w_i}$.
\end{proof}

From now on, whenever special K\"ahler coordinates $(z^1 , \dots , z^{g_S})$ are being used we will let $\partial_i$ denote $\frac{\partial}{\partial z^i}$.

\begin{lemma}
Let $(z^1,\dots , z^{g_S})$ be local special K\"ahler coordinates on $B^{\rm reg}$. Then (as cohomology classes) we have:
\[
\nabla_{\partial_i } \theta = \omega_i.
\]
\end{lemma}
\begin{proof}
Recall that for any vector field $X$ on $B^{\rm reg}$ we have $\phi(X) = \nabla_X \theta$ where $\phi(X)$ is a $(1,0)$-form. We can determine the $1$-form by integrating against $a$-cycles:
\[
\int_{a_j} \nabla_{\partial_i } \theta = \partial_i \int_{a_j} \theta = \partial_i(z^j) = \delta_{ij}.
\]
Hence $\nabla_{\partial_i } \theta = \omega_i$ as claimed.
\end{proof}

\begin{proposition}
Let $X,Y,Z$ be vector fields on $B^{\rm reg}$. We have:
\begin{eqnarray}
\int_S \nabla_X \theta \wedge \theta &=& 0, \label{equ:d1}\\ 
\int_S \nabla_X \nabla_Y \theta \wedge \theta &=& 0, \label{equ:d2}\\
\int_S \nabla_X \nabla_Y \nabla_Z \theta \wedge \theta &=& c(X,Y,Z) \label{equ:d3},
\end{eqnarray}
where $c(X,Y,Z)$ is the Donagi-Markman cubic.
\end{proposition}
\begin{remark}
The expression $\int_S \nabla_X \nabla_Y \nabla_Z \theta \wedge \theta$ is the analogue of Yukawa couplings for moduli spaces of Calabi-Yau $3$-folds.
\end{remark}
\begin{proof}
Recall as above that $\phi(X) = \nabla_X \theta$ is a $(1,0)$-form. Thus $\nabla_X \theta \wedge \theta= 0$, since it is a $(2,0)$-form on $S$. This proves (\ref{equ:d1}). Applying $\nabla_Y$ to (\ref{equ:d1}), we get:
\[
0 = Y \left( \int_S \nabla_X \theta \wedge \theta \right) = \int_S \nabla_Y \nabla_X \theta \wedge \theta + \int_S \nabla_X \theta \wedge \nabla_Y \theta.
\]
But $\int_S \nabla_X \theta \wedge \nabla_Y \theta = 0$, as $\nabla_X\theta $ and $\nabla_Y \theta$ are both $(1,0)$-forms, so we get (\ref{equ:d2}). Observe now that the left hand side of (\ref{equ:d3}) is a symmetric cubic tensor because of (\ref{equ:d2}). Therefore it suffices to consider the case where $X = \partial_i$, $Y = \partial_j$, $Z = \partial_k$. In this case we get
\begin{equation*}
\begin{aligned}
\int_S  \nabla_{\partial_i} \nabla_{\partial_j} \nabla_{\partial_k} \theta \wedge \theta &= \sum_{l = 1}^{g_S} \left( \int_{a_l} \nabla_{\partial_i} \nabla_{\partial_j} \nabla_{\partial_k} \theta \int_{b_l} \theta  - \int_{b_l} \nabla_{\partial_i} \nabla_{\partial_j} \nabla_{\partial_k} \theta \int_{a_l} \theta  \right) \\
&= \sum_{l=1}^{g_S} \left( w_l \partial_i \partial_j \partial_k \int_{a_l} \theta - z^l \partial_i \partial_j \partial_k \int_{b_l} \theta \right) \\
&= \sum_{l=1}^{g_S} \left( w_l \partial_i \partial_j \partial_k z^l  - z^l \partial_i \partial_j \partial_k w_l \right)\\
&= -\sum_{l=1}^{g_S} z^l \partial_i \partial_j \tau_{kl} \\
&= -\sum_{l=1}^{g_S} z^l \partial_l \partial_j \tau_{ki} \quad \text{( by symmetry of }  \partial_j \tau_{kl}) \\
&= -\xi \partial_j \partial_k w_i .
\end{aligned}
\end{equation*}
Then using the commutation relation $[\xi , \partial_i] = -\partial_i$ and the fact that $\xi(w_i) = w_i$, we see that $-\xi \partial_j \partial_k w_i = \partial_j \partial_k w_i = \partial_j \tau_{ki} = c_{jki} = c_{ijk}$.
\end{proof}

\begin{proposition}
A prepotential for the special K\"aher structure on $B^{\rm reg}$ is given by
\[
\mathcal{F} = \frac{1}{2} z^i w_i = \frac{1}{2} \tau_{ij} z^i z^j.
\]
\end{proposition}
\begin{proof}
This is a simple calculation:
\[
\partial_i \mathcal{F} = \frac{1}{2}w_i + \frac{1}{2} z^j \partial_i w_j = \frac{1}{2}w_i + \frac{1}{2}z^j \tau_{ij} = \frac{1}{2}w_i + \frac{1}{2}w_i = w_i.
\]
\end{proof}

\section{Donagi-Markman cubic}\label{sec:dmcubic}

The Donagi-Markman cubic measures the change in periods of the spectral curve $S_b$ as $b$ varies. Since the periods are essentially given by the Hodge structure on $H^1(S_b , \mathbb{C})$, this amounts to computing the variation of Hodge structure, which is controlled by the Kodaira-Spencer class of the deformation of $S_b$. Therefore, we need to compute the Kodaira-Spencer class $\kappa(\partial) \in H^1(S_b , T_{S_b})$ associated to a tangent vector $\partial \in T_b B^{\rm reg}$.

\subsection{Kodaira-Spencer class computation}
Recall the characteristic map $\chi : T_b B^{\rm reg} \to H^0(S_b , N_{S_b})$ in (\ref{equ:charmap}) which sends a tangent vector in $T_b B^{\rm reg}$ to the corresponding normal vector field on $S_b$. Let us write $V = \chi(\partial)$ for the normal vector field corresponding to $\partial$. Next, recall the universal moduli spaces of spectral curves:
\[
Z^{\rm reg} = \{ (x,b) \in T^*\Sigma \times B^{\rm reg} \; | \; p_b(x) = 0\}.
\]
Let $q : Z^{\rm reg} \to B^{\rm reg}$ be the natural projection and $j : Z^{\rm reg} \to T^*\Sigma$ the map sending a spectral curve to its image in $T^*\Sigma$. We denote the composition of $j$ with the natural projection $T^*\Sigma \to \Sigma$ simply as $\pi : Z^{\rm reg} \to \Sigma$. For any $b \in B$ we can find an open neighbourhood $b \in U \subset B^{\rm reg}$ over which the family $Z^{\rm reg}$ can be differentiably trivialised:
\[
Z^{\rm reg}|_U = q^{-1}(U) \cong U \times S,
\]
where $S = S_b$. With respect to this trivialisation, the family $Z^{\rm reg}|_U$ consists of a fixed topological surface $S$ with a complex structure $I(t)$ that varies with $t \in U$. The map $\pi : Z^{\rm reg}|_U \to \Sigma$ corresponds to a family of maps $\pi_t : S \to \Sigma$ such that $\pi_t$ is holomorphic with respect to $I(t)$. Given a tangent vector $\partial \in T_bU$, we use a dot to denote differentiation by $\partial$. Differentiating the condition that $\pi_t$ is holomorphic with respect to $I(t)$, we get
\begin{equation}\label{equ:kap1}
\pi_*( \kappa(\partial) ) = \overline{\partial} Y,
\end{equation}
where $\kappa(\partial) = -\frac{i}{2}\dot{I}$ is the Kodaira-Spencer class of the deformation of $I$ in the direction $\partial$ and $Y = \dot{\pi} \in \Omega^{0}(S , \pi^*(T_\Sigma) )$ is the corresponding deformation of the map $\pi$. If $a \in S$ is a ramification point of $\pi$ then $Y(a) \in ( T_\Sigma)_{\pi(a)}$ is a tangent vector describing the motion of the branch point $\pi(a) \in \Sigma$. Let $D$ be the ramification divisor of $\pi$, that is $D = \sum_a (r(a)-1)a$ where the sum is over ramification points of $\pi$ and $r(a)$ is the ramification degree of $\pi$ at $a$. View $\pi_*$ as a section of $T_S^{-1} \otimes \pi^*(T_\Sigma)$ and define
\[
W = \frac{Y}{\pi_*} \in \Omega^{0}( S , T_S(D) ),
\]
which is a vector field on $S$ having poles at the ramification points. Then Equation (\ref{equ:kap1}) says
\[
\kappa(\partial) = \overline{\partial} W.
\]
Note that $W$ depends on $\partial$ and we will write $W(\partial)$ if we wish to show this dependence. Note also that $W$ depends on the choice of local differentiable trivialisation of the family of spectral curves.

\begin{definition}\label{def:delta}
Let $\partial$ be a $(1,0)$-vector field on $B^{\rm reg}$. We denote by $\delta$ the unique vector field on $Z^{\rm reg} \setminus D$ which is a lift of $\partial$ (that is, $q_*(\delta) = \partial$), satisfying $\pi_* (\delta ) = 0$ (i.e., $\pi$ is constant along the integral curves of $\delta$). In a local differentiable trivialisation $Z^{\rm reg}|_U \cong U \times S$, $\delta$ is given by:
\begin{equation}\label{equ:liftingvf}
\delta = \partial - W(\partial) \in \Omega^0( Z^{\rm reg} , TZ^{\rm reg}(D) ).
\end{equation}
Note that $\delta$ aquires poles along the ramification divisor $D$ and so it may be regarded as a section of $TZ^{\rm reg}(D)$.
\end{definition}

We use $\delta$ to differentiate objects on the family $Z^{\rm reg}$ in a trivialisation-independent manner. The only drawback of this is that differentiation with respect to $\delta$ produces poles at the ramification points.
\begin{lemma}\label{lem:deltatheta}
We have $\delta \theta = i_V d\theta |_S$, where $V = \chi(\partial)$ is the normal vector field corresponding to $\partial \in TB^{\rm reg}$.
\end{lemma}
\begin{proof}
Choose a local differentiable trivialisation of the family of spectral curves: $Z^{\rm reg}|_U \cong U \times S$. Define $\tilde{V} = j_*(\partial) \in \Omega^0( S , T_{T^*\Sigma} )$ and $\hat{V} = j_*(\delta) \in \Omega^0( S , T_{T^* \Sigma}(D))$. By this definition, $\tilde{V}$ and $\hat{V}$ are lifts of $V = \chi(\partial)$. Next, we note that since $\delta = \partial - W$ preserves $\pi \circ j : Z^{\rm reg}|_U \to \Sigma$, we have that $j_*(\delta)$ preserves $\pi$ in the sense that $\pi_* j_* \delta = 0$. Thus, $\pi_* \hat{V} = 0$. Observing that $j_* : T(U \times S) \to T_{T^*\Sigma}$ restricted to $T_S$ acts as the identity $j_* : T_S \to T_S$, we find $\hat{V} = j_*(\delta) = j_*(\partial - W) = \tilde{V} - W$. We see that $\hat{V}$ is a lift of $V$ to a section of $T_{T^*\Sigma}(D)$ which is required to satisfy $\pi_* \hat{V} = 0$. Applying $\overline{\partial}$ to $\pi_* \tilde{V} = \pi_* W$, we get $\pi_*( \overline{\partial} \tilde{V} ) = \pi_* (\overline{\partial} W) = \pi_* \kappa$, but $\pi_* : T_S \to \pi^* T_\Sigma$ is generically an isomorphism, so we deduce that $\kappa = \overline{\partial} W = \overline{\partial} \tilde{V}$ and $\hat{V}$ is meromorphic. The identity $j_* (\partial) = \tilde{V}$ means that our local differentiable trivialisation of the family of spectral curves is given by integrating the flow lines of $\tilde{V}$. This means that in our given trivialisation, the change in $\theta$ is $\partial \theta = \mathcal{L}_{\tilde{V}} \theta |_S$. Then:
\[
\delta \theta = \partial \theta - \mathcal{L}_W \theta = \mathcal{L}_{\tilde{V} - W}(\theta)|_S = \mathcal{L}_{ \hat{V} } \theta |_S = i_{\hat{V}} d\theta + d( i_{\hat{V}} \theta) |_S.
\]
But $\theta$ vanishes at the ramification points with order at least that of $\pi_*$, so $i_{\hat{V}}\theta $ is a holomorphic function, hence constant on $S$. Therefore $d(i_{\hat{V}} \theta )|_S = 0$ and $\delta \theta = i_{\hat{V}} d\theta |_S$. But $\hat{V}$ is a lift of $V$, so $i_{\hat{V}} d\theta |_S = i_{\tilde{V}} d\theta |_S$.
\end{proof}

Consider the natural $\mathbb{C}^*$-action on $T^*\Sigma$ given by scalar multiplication in the fibres. Let $\xi \in H^0( T^*\Sigma , T_{T^*\Sigma} )$ be the vector field generating this action. In local coordinates $(x,y)$, where $x$ is a local coordinate on $\Sigma$ and $(x,y)$ corresponds to the point $(x , ydx) \in T^*\Sigma$, we have $\xi = y \frac{\partial}{\partial y}$. Note that in these coordinates the canonical $1$-form is given by $\theta = y dx$, hence $\theta(\xi) = 0$, $i_\xi d\theta = \theta$.

\begin{lemma}\label{lem:hatv2}
Let $V , \hat{V}$ be as in Lemma \ref{lem:deltatheta}. Then $\hat{V}$ is given by:
\begin{equation*}
\hat{V} = \frac{ \alpha }{\theta} \xi,
\end{equation*}
where $\alpha = i_V d\theta |_{S_b}$.
\end{lemma}
\begin{proof}
Recall that $\hat{V}$ is a lift of $V$ satisfying $\pi_*(\hat{V}) = 0$. These conditions uniquely determine $\hat{V}$, because $Ker(\pi_*) \cap T_S$ is generically zero. Let $V^* = \frac{\alpha}{\theta} \xi$. By uniqueness, it is enough to check that $V^*$ satisfies the necessary requirements, i.e. $V^*$ is a section of $T_{T^*\Sigma}(D)$ which is a lift of $V$ and satisfies $\pi_* V^* = 0$. Clearly $\pi_* V^* = 0$, because $\xi \in Ker(\pi_*)$. We need to check that $V^*$ is a section of $T_{T^*\Sigma}(D)$. Let $(x,y)$ be local coordinates on $T^*\Sigma$ as above and let $q$ be a local coordinate on $S_b$. Then $\pi$ is the map $\pi(q) = x = x(q)$ and with respect to these coordinates $d\pi = \frac{dx}{dq}$. The $1$-form $\alpha$ has the form $\alpha(q) dq$ for some holomorphic function $\alpha(q)$. Then:
\begin{equation*}
\begin{aligned}
V^* = \frac{\alpha}{\theta} \xi &= \frac{ \alpha(q) dq }{ ydx } y \frac{\partial}{\partial y}\\
&= \frac{ \alpha(q) dq }{ dx } \frac{\partial}{\partial y} \\
&= \frac{ \alpha(q) }  { (\frac{dx}{dq}) } \frac{\partial}{\partial y} = \frac{ \alpha(q) }  { d\pi } \frac{\partial}{\partial y}.
\end{aligned}
\end{equation*}
So the polar divisor of $V^*$ is at most $D$. To show that $V^*$ projects to $V$, we just need to show that
\[
i_{V^*} d\theta = i_V d\theta = \alpha,
\]
where the second equality is the definition of $\alpha$. This follows easily from $i_\xi d\theta = \theta$, indeed:
\[
i_{V^*} d\theta = \frac{\alpha}{\theta} i_{\xi} d\theta = \frac{\alpha}{\theta} \theta = \alpha.
\]
Hence $V^* = \hat{V}$.
\end{proof}

The following result relates the characteristic map to the Gauss-Manin connection:
\begin{lemma}\label{lem:nablachi}
For any $\partial \in T_b B^{\rm reg}$, we have:
\[
\nabla_\partial \theta = \delta \theta = i_{\chi(\partial)} d\theta |_{S_b}.
\]
More precisely, the right hand side is the unique holomorphic $1$-form representing the cohomology class $\nabla_\partial \theta \in H^1( S_b , \mathbb{C})$.
\end{lemma}
\begin{proof}
The second equality is Lemma \ref{lem:deltatheta}. Moreover the right hand side is clearly holomorphic. Thus it remains only to show that $\nabla_\partial \theta = i_{\chi(\partial)} d\theta |_{S_b}$, at the level of cohomology classes. For this we view $B^{\rm reg}$ as a moduli space of spectral curves, which are complex Lagrangians in the symplectic surface $(T^*\Sigma , d\theta)$. The $\mathcal{V}_\mathbb{C}$-valued $1$-form $\phi$ is essentially by definition given by $\phi(X) = i_{\chi(X)} d\theta |_{S_b}$ (this is just a special case of Equation (\ref{equ:phi})). On the other hand $\phi = d_\nabla \theta$, and the lemma follows.
\end{proof}

\begin{remark}\label{rem:vhat}
By Lemma \ref{lem:nablachi}, we have $\alpha = \nabla_\partial \theta$ and so $\hat{V}$ can be written as
\[
\hat{V} = \frac{\nabla_\partial \theta}{\theta} \xi = \frac{\delta \theta}{\theta} \xi,
\]
where it is understood that by $\nabla_\partial \theta$, we mean the holomorphic $1$-form representing $\nabla_\partial \theta$ in cohomology.
\end{remark}

\subsection{Residue formula for the Donagi-Markman cubic}

\begin{theorem}\label{thm:dmcub}
The Donagi-Markman cubic is given by the following residue formula:
\[
c(X,Y,Z) = -2\pi i \sum_a \underset{a}{Res} \left( \frac{ (\nabla_X \theta)(\nabla_Y \theta )(\nabla_Z \theta)\xi}{\theta} \right),
\]
where the sum is over the ramification points of $\pi$. With respect to local special K\"ahler coordinates $(z^1 , \dots , z^{g_S})$ this expression takes the form
\begin{equation}\label{equ:cijk}
c_{ijk} = -2\pi i \sum_a \underset{a}{Res} \left( \frac{ \omega_i \omega_j \omega_k \xi }{\theta} \right).
\end{equation}
\end{theorem}

\begin{remark}\label{rem:inter}
Before giving the proof we should comment on how to interpret the right hand side of this formula. The vector field $\xi$ is a section of $T_{T^*\Sigma}|_{S_b}$ which in general is not tangent to $S_b$. However, it is precisely at the ramification points where we have that $\xi$ is tangent to $S_b$. Let $\zeta$ be a meromorphic section of $T_S^{2}$ defined in a neighbourhood of $a$ and such that, when viewed as a section of $T_S \otimes T_{T^*\Sigma}$, the difference $\frac{\xi}{\theta} - \zeta$ is holomorphic. Then the right hand side of the formula can be read as:
\[
-2\pi i \sum_a \underset{a}{Res} \left( (\nabla_X \theta)(\nabla_Y \theta )(\nabla_Z \theta) \zeta \right). 
\]
Note that such a $\zeta$ always exists. Indeed, in local coordinates $(x,y)$ on $T^*\Sigma$, $\frac{\xi}{\theta} = \frac{ y \partial_y}{y dx} = \frac{\partial_y}{dx}$. If $a$ is a ramification point of degree $r(a)$ then $dx|_S$ vanishes to order $r(a)-1$. But $y(a)$ is a zero of the characteristic polynomial $p(y,x)$ of order $r(a)$, so $(\partial^j_y p)(a) = 0$ for $j = 0 , \dots , r(a)-1$. In other words, $\partial_y$ is tangent to $S$ at $a$ to order $r(a)-1$, hence the polar part of $\frac{\xi}{\theta}$ is tangent to $S$. We will give a more explicit formula for this expression below, in the case where $\pi$ has simple branching only (Remark \ref{rem:zeta}).
\end{remark}
\begin{proof}
We start with Equation (\ref{equ:d2}):
\[
\int_S \nabla_Y \nabla_Z \theta \wedge \theta = 0.
\]
Applying $\nabla_X$ then gives
\[
\int_S \nabla_X \nabla_Y \nabla_Z \theta \wedge \theta + \int_S \nabla_Y \nabla_Z \theta \wedge \nabla_X \theta = 0.
\]
Combined with Equation (\ref{equ:d3}), we get:
\[
c(X,Y,Z) = -\int_S \nabla_Y \nabla_Z \theta \wedge \nabla_X \theta.
\]
Since $\nabla_X \theta$ is a $(1,0)$-form, this integral only depends on the $(0,1)$-part of $\nabla_Y \nabla_Z \theta$. By the Griffiths transversality theorem for variations of Hodge structure we know that:
\[
( \nabla_Y \nabla_Z \theta)^{(0,1)} = \kappa(Y) \cup \nabla_Z \theta \in H^1(S , \mathcal{O}),
\]
where $\cup$ denotes the cup product $\cup : H^1(S , T_S ) \otimes H^0(S , K_S) \to H^1(S , \mathcal{O})$. Thus
\[
c(X,Y,Z) = -\int_S \left( \kappa(Y) \cup \nabla_Z \theta \right) \wedge \nabla_X \theta.
\]
Let $V = \chi(Y)$, and let $\tilde{V}$ be a lift of $V$ to a smooth section of $T_{T^*\Sigma}|_{S_b}$. Let $\hat{V}$ be as in Lemma \ref{lem:hatv2}, which by the Remark \ref{rem:vhat} is given by
\[
\hat{V} = \frac{\nabla_Y \theta}{\theta} \xi.
\]
Recall that $W = \tilde{V} - \hat{V}$ is a smooth section of $T_{S_b}(D)$ and that $\overline{\partial}(W) = \overline{\partial}( \tilde{V} )$ is a representative for the Kodaira-Spencer class $\kappa(Y)$. Let $\zeta$ be as in Remark \ref{rem:inter}. It follows that $W = -(\nabla_Y \theta) \zeta + W'$, where $W'$ is smooth. We then have
\begin{equation*}
\begin{aligned}
c(X,Y,Z) &= -\int_S \left( \overline{\partial} W \cup \nabla_Z \theta \right) \wedge \nabla_X \theta \\
&= -\int_S \overline{\partial} \left(  i_W ( \nabla_Z \theta) \wedge \nabla_X \theta \right) \\
&= -2\pi i \sum_a \underset{a}{Res} \left(  i_{(\nabla_Y \theta)\zeta} ( \nabla_Z \theta) \wedge \nabla_X \theta \right) \\
&= -2 \pi i \sum_a \underset{a}{Res} \left( (\nabla_X \theta)(\nabla_Y \theta )(\nabla_Z \theta)\zeta \right), \\
&= -2 \pi i \sum_a \underset{a}{Res} \left( \frac{ (\nabla_X \theta)(\nabla_Y \theta )(\nabla_Z \theta)\xi}{\theta} \right),
\end{aligned}
\end{equation*}
where the sum is over the ramification points of $S$.
\end{proof}

\begin{remark}\label{rem:zeta}
We may re-write Equation \ref{equ:cijk} in terms of local coordinates as follows. For this, we will assume {\em $\pi$ has only simple ramification points}. Let $a \in S$ be a ramification point. Let $x$ be a local coordinate on $\Sigma$ centered at $\pi(a)$, $q$ a local coordinate on $S$ centered at $a$ chosen so that $\pi(q) = x = q^2$. Lastly, let $(x,y)$ be local coordinates on $T^*\Sigma$ with $\theta = y dx$ the canonical $1$-form and $\xi = y \partial_y$. At the point $a$, we have $\partial_y = \frac{dq}{dy}\partial_q$. Therefore we may choose $\zeta$ in Remark \ref{rem:inter} to be
\begin{equation}\label{equ:zeta}
\zeta = \frac{ y \frac{dq}{dy} \partial_q }{ y dx } = \frac{1}{dx dy }.
\end{equation}
Then the residue contribution to $c_{ijk}$ at $a$ is:
\[
\underset{q \to 0}{Res} \left( \omega_i \omega_j \omega_k \zeta \right) = \underset{q\to0}{Res}\left( \frac{ \omega_i \omega_j \omega_k}{dx dy } \right).
\]
We will assume that coordinates $x,q$ have been chosen around each ramification point and so we will write Equation (\ref{equ:cijk}) as:
\[
c_{ijk} = -2\pi i \sum_a \underset{a}{Res} \left( \frac{\omega_i \omega_j \omega_k}{dx dy} \right).
\]
\end{remark}

\section{Relation to topological recursion}\label{sec:toprec}

\subsection{Topological recursion for Hitchin spectral curves}

Topological recursion, as introduced in \cite{EO}, is a recursive procedure which takes a Riemann surface $S$ (we assume $S$ is compact) with meromorphic functions $x,y$ such that $dx$ has only simple zeros, and produces a collection of symmetric multidifferentials $W^{(g)}_n$, for $g \ge 0$, $n \ge 1$, known as the {\em Eynard-Orantin invariants}. By ``multidifferential" we mean that $W^{(g)}_n$ is a meromorphic section of the $n$-th exterior tensor product $K_S^{\boxtimes n} = K_S \boxtimes K_S \boxtimes \cdots \boxtimes K_S$ on $S^n$. By symmetric we mean invariant under the action of the permutation group on $S^n$. The recursion formula requires a choice of symplectic basis $a_1 , \dots , a_{g_S} , b_1 , \dots , b_{g_S}$. The two base cases\footnote{An alternative convention is to define $W^{(0)}_1 = y dx$, but all other $W^{(g)}_n$'s are unchanged.} are:
\begin{equation*}
\begin{aligned}
W^{(0)}_1(p) &= 0, \\
W^{(0)}_2(p_1,p_2) &= B(p_1,p_2),
\end{aligned}
\end{equation*}
where $B(p_1,p_2)$ is the Bergman kernel\footnote{One can extend topological recursion by replacing $B$ with a modified version of the Bergman kernel, however we will not make use of this generalisation.} on $S \times S$. Recall (eg, \cite{Fay0}) that this is the unique meromorphic symmetric bi-differential which is holomorphic away from the diagonal and which, in a local coordinate $q(p)$, has an expansion around the diagonal of the form:
\[
B( q(p_1) , q(p_2) ) = \frac{dq(p_1)dq(p_2)}{(q(p_1)-q(p_2))^2} + O(1) dq(p_1)dq(p_2)
\]
and such that $B$ is normalised with respect to the symplectic basis in the sense that $\int_{p_1 \in a_i} B(p_1 , p_2) = 0$ for $i = 1 , \dots , g_S$. To define the recursion formula we need some further notation. By assumption the map $x : S \to \mathbb{P}^1$ is simply branched. Thus, for any ramification point $a \in S$, we can find a neighbourhood $a \in U \subset S$ and a non-trivial involution $\sigma : U \to U$ such that $x \circ \sigma = x$. If $q \in U$ one writes $\overline{q} = \sigma(q)$. Thus $x(q) = x(\overline{q})$. Let $\omega(q)$ denote the $1$-form on $U$ given by
\[
\omega(q) = ( y(q) -  y(\overline{q}) dx(q) = (\theta - \sigma^*\theta)(q),
\]
where $\theta = y dx$. For $p \in S$ and $q \in U$, we define $dE_q(p)$ as follows. We assume the neighbourhood $U$ is chosen to be simply-connected and set
\[
dE_q(p) = \frac{1}{2} \int_{\xi \in q}^{\overline{q}} B(\xi , p),
\]
where the path of integration is taken to lie in $U$. Clearly this is independent of the choice of path of integration. 

Now we are ready to state the recursion formula. Let $p_1, \dots , p_n$ be points on $S$. If $K = \{ i_1 , \dots , i_k\}$ is a subset of $\{1,2, \dots , n\}$, we let $p_K$ be the $k$-tuple $p_K = (p_{i_1} , p_{i_2} , \dots , p_{i_k})$. Then for $2g - 2 + k > 0$ we define
\begin{equation}\label{equ:toprec}
\begin{aligned}
&W^{(g)}_{k+1}(p,p_K) = \\
& \underset{q \to a}{Res} \frac{dE_q(p)}{\omega(q)} \left( \sum_{m=0}^{g} \sum_{J \subseteq K} W^{(m)}_{|J|+1}(q , p_J) W^{(g-m)}_{k-|J|+1}(\overline{q} , p_{K \setminus J} ) + W^{(g-1)}_{k+2}(q,\overline{q} , p_K) \right)
\end{aligned}
\end{equation}
where the sum $\sum_{J \subseteq K}$ is over all subsets $J \subseteq K$. Notice that all non-zero terms on the right hand side involve only terms $W^{(g')}_{k'}$ with $2g'-2+k' < 2g-2+k$. Therefore, this gives a recursive definition of the $W^{(g)}_k$.

The topological recursion formula was adapted to the case of Hitchin spectral curves in \cite{dumu}. Here the map $x : S \to \mathbb{P}^1$ is replaced by $\pi : S \to \Sigma$. To make sense of the recursion formula (\ref{equ:toprec}) on a spectral curve $S \subset T^*\Sigma$, note that the formula does not directly involve the functions $x,y$ only the $1$-form $\theta = y dx$. For the recursive formula, we only need the Bergman kernel $B$, the local involutions $\sigma$ about each ramification point, and the $1$-forms $\omega = \theta - \sigma^* \theta$, defined around each ramification point. The local involutions $\sigma$ are well-defined in a neighbourhood of each ramification point {\em provided $\pi : S \to \Sigma$ has only simple branching}. We will assume for the rest of this paper that this is the case.

\subsection{Variational formulas}

Our goal in this section is to relate the $g=0$ Eynard-Orantin invariants of spectral curves to the special K\"ahler geometry of $B^{\rm reg}$. The key result which ties these together is the following variational formula for the Bergman kernel:
\begin{proposition}[Rauch variational formula]\label{prop:rvf}
Let $\partial \in T_b B^{\rm reg}$. Assume $p,r$ are distinct and are not ramification points. Then:
\[
\delta B( p, q ) = -\sum_a \underset{u \to a}{Res} \; \frac{ \delta \theta(u) B(u,p)B(u,r) }{ dx(u) dy(u) },
\]
where the sum is over the ramification points of $\pi$ and for each ramification point $a \in S_b$, we choose coordinate functions $x,q$ as in Remark \ref{rem:zeta}.
\end{proposition}
\begin{proof}
Although this formula is well known, we were unable to find a satisfactory proof in the literature that applies to our setting, so we provide a proof here. Choose a local differentiable trivialisation in a neighborhood $U$ of $b$: $Z^{\rm reg}|_U \cong U \times S$, so that $\delta = \partial - W$. Since $\delta$ is independent of the choice of local trivialisation, we are free to choose such a trivialisation at our convenience. Changing a given local trivialisation by a suitably chosen diffeomorphism yields a change in $W$ of the form $W \mapsto W + X$, where $X$ is an arbitrary smooth vector field on $S$ ($X$ has no poles). From this it is clear that, for a fixed choice of points $p,r$ distinct from the ramification points, we can assume $W$ vanishes in a neigbourhood of $p$ and $r$. Then
\begin{equation*}
\delta B(p,r) = \partial B(p,r) - \mathcal{L}_{W(p)} B(p,r) - \mathcal{L}_{W(r)} B(p,r) = \partial B(p,r),
\end{equation*}
because $W$ vanishes around $p$ and $r$. Next, we have the following variational formula for $B(p,r)$ \cite[page 57]{Fay}:
\begin{equation*}
\begin{aligned}
\partial B(p,r) &= \kappa(p) B(p,r) + \kappa(r) B(r,p) \\
& \quad \quad - \frac{1}{2\pi i} \text{ p.v.} \int_S ( \kappa( \, \cdot \, ) B( \, \cdot \,  , p ) ) \wedge B( \, \cdot \, , r ), \\
\end{aligned}
\end{equation*}
where $\text{p.v.}$ denotes the Cauchy principal value of the integral. Using $\kappa = \overline{\partial} W$, we obtain
\begin{equation*}
\begin{aligned}
\partial B(p,r) &= \kappa(p) B(p,r) + \kappa(q) B(r,p) \\
& \quad \quad + \frac{1}{2\pi i} \sum_a  \int_{u \in \gamma_a} W( u ) B( u  , p )  B( u , r ) , \\
&= \frac{1}{2\pi i} \sum_a  \int_{u \in \gamma_a} W( u ) B( u  , p )  B( u , r ) ,
\end{aligned}
\end{equation*}
where we obtain the last line because $\kappa = 0$ at $p$ and $r$ by our assumption on $W$. The sum $\sum_a$ is taken over all poles of $W( u ) B( u  , p ) B( u , r )$, namely $a$ is a ramification point, $a = p$ or $a = r$ and $\gamma_a$ is a contour around the point $a$. However, we again have by our assumptions on $W$ that there are no residue contributions from the points $p,r$ and therefore we can take the sum to just be over the ramification points of $\pi$. Now, as in the proof of Theorem \ref{thm:dmcub}, we may write $W = -(\delta \theta) \zeta + W'$, where $W'$ is smooth and $\zeta$ is given by (\ref{equ:zeta}). Then
\begin{equation*}
\begin{aligned}
\delta B(p,r) &= \frac{1}{2\pi i} \sum_a  \int_{u \in \gamma_a} (-(\delta \theta)(u) \zeta(u) + W'(u) ) B( u  , p )  B( u , r ) ,\\
&= -\sum_a \underset{u \to a}{Res} \; (\delta \theta)(u) \zeta(u) B(u,p)B(u,r) \\
&= -\sum_a \underset{u \to a}{Res} \;  \frac{ \delta \theta(u) B(u,p)B(u,r)}{dx(u) dy(u)}.
\end{aligned}
\end{equation*}
\end{proof}

Let $(z^1 , \dots , z^{g_S})$ be local special K\"ahler coordinates on $B^{\rm reg}$. Let $\partial_i$ denote the vector field $\frac{\partial}{\partial z^i}$ and let $\delta_i$ denote $\delta( \partial_i)$. We have:
\begin{theorem}[Variational formula \cite{EO}]
For $g+k > 1$, 
\[
\delta_i W^{(g)}_k(p_1 , \dots , p_k) = -\frac{1}{2\pi i} \int_{p \in b_i} W^{(g)}_{k+1}(p, p_1 , \dots , p_k).
\]
\end{theorem}
\begin{proof}
This is essentially Theorem 5.1 of \cite{EO}. Since we are working in a setting where $\theta$ is not globally of the form $\theta = y dx$, one needs to check the proof of Theorem 5.1 in \cite{EO} holds in this setting. In fact the proof in \cite[pages 32-34]{EO} essentially only relies on the Rauch variational formula (which we have proven in Proposition \ref{prop:rvf}) and the diagrammatic representation of $W^{(g)}_k$ \cite[Theorem 4.8]{EO}, the proof of which only uses the recursive definition of $W^{(g)}_k$ and does not involve any global properties of the spectral curve.
\end{proof}

\begin{remark}
Note that the poles of $W^{(g)}_k$, in any one of its variables, have zero residues. This can easily be deduced from symmetry and the diagrammatic representation of $W^{(g)}_k$ \cite[Theorem 4.8]{EO}. Therefore the integration of $W^{(g)}_k$ over a cycle $\gamma$ (chosen so as to avoid the poles) depends only on the homology class of $\gamma$ in $S$. Similarly one can also show that the integration of $W^{(g)}_k$ over an $a$-cycle is zero.
\end{remark}

Consider the case $(g,k) = (0,2)$, where $W^{(0)}_2(p_1,p_2) = B(p_1,p_2)$ is the Bergman kernel. Applying the variational formula, we obtain
\[
\delta_i B(p_1,p_2) = -\frac{1}{2\pi i} \int_{p \in b_i} W^{(0)}_{3}(p, p_1 , p_2).
\]
Recall that $\int_{p_1 \in b_j} B(p_1,p_2) = 2 \pi i \omega_j(p_2)$ \cite{Fay0} and thus $\int_{p_1 \in b_j} \int_{p_2 \in b_k} B(p_1,p_2) = 2\pi i \tau_{jk}$. The variational formula then gives
\begin{equation*}
\begin{aligned}
c_{ijk} = \partial_i \tau_{jk} &= \frac{1}{2\pi i} \partial_i \int_{p_1 \in b_j} \int_{p_2 \in b_k} B(p_1,p_2) \\
&= \frac{1}{2\pi i} \int_{p_1 \in b_j} \int_{p_2 \in b_k} \delta_i B(p_1,p_2) \\
&= -\left(\frac{1}{2\pi i}\right)^2 \int_{p \in b_i} \int_{p_1 \in b_j} \int_{p_2 \in b_k} W^{(0)}_3(p, p_1,p_2).
\end{aligned}
\end{equation*}
From \cite[Theorem 4.1]{EO}, we have:
\[
W^{(0)}_3(p,p_1,p_2) = \sum_a \underset{q \to a}{Res} \; \left( \frac{ B(p,q)B(p_1,q)B(p_2,q) }{dx(q)dy(q)} \right).
\]
Therefore, we obtain
\begin{equation*}
\begin{aligned}
c_{ijk} &= -\left(\frac{1}{2\pi i}\right)^2  \sum_a \int_{p \in b_i} \int_{p_1 \in b_j} \int_{p_2 \in b_k} \underset{q \to a}{Res} \; \left( \frac{ B(p,q)B(p_1,q)B(p_2,q) }{dx(q)dy(q)} \right) \\
&= -2\pi i \sum_a \underset{q \to a}{Res} \; \left( \frac{ \omega_i(q)\omega_j(q)\omega_j(q) }{dx(q)dy(q)} \right),
\end{aligned}
\end{equation*}
which agrees with Theorem \ref{thm:dmcub} (see Remark \ref{rem:zeta}). In a similar manner, starting with $W^{(0)}_2(p_1,p_2) = B(p_1,p_2)$ and applying the variational formula multiple times, we obtain:
\begin{theorem}\label{thm:derivtau}
We have:
\[
\partial_{i_1} \partial_{i_2} \cdots \partial_{i_{m-2}} \tau_{ i_{m-1} i_m } = -\left( \frac{i}{2\pi} \right)^{m-1} \int_{p_1 \in b_{i_1}} \cdots \int_{p_m \in b_{i_m} } W^{(0)}_m(p_1 , \dots , p_m).
\]
\end{theorem}
Therefore, the $g=0$ Eynard-Orantin invariants $W^{(0)}_k$ for a spectral curve $S_b$ compute the power series expansion of the period matrix $\tau_{ij}$ about $b \in B^{\rm reg}$. Since the special K\"ahler metric on $B^{\rm reg}$ is given in terms of the period matrix, the invariants $W^{(0)}_k$ also compute the power series expansion of the special K\"ahler metric.

\subsection{Second derivatives of the period matrix by topological recursion}

We will use Theorem \ref{thm:derivtau} to compute the symmetric quartic $\partial_{i} \partial_{j} \tau_{kl}$ of second derivatives of the period matrix. From the diagrammatic representation of Eynard-Orantin invariants, one finds (\cite[Equation (4-46)]{EO}):
\begin{equation*}
\begin{aligned}
W^{(0)}_4(p,p_1,p_2,p_3) &= \sum_{a,b} \underset{q \to a}{Res} \underset{r \to b}{Res} \frac{dE_q(p)}{\omega(q)}\frac{dE_r(q)}{\omega(r)} \left[ B(\overline{q},p_1)B(r,p_2)B(\overline{r},p_3) + \text{perm}_{1,2,3} \right] \\
& + \sum_{a,b} \underset{q \to a}{Res} \underset{r \to b}{Res} \frac{dE_q(p)}{\omega(q)}\frac{dE_r(\overline{q})}{\omega(r)} \left[ B(q,p_1)B(r,p_2)B(\overline{r},p_3) + \text{perm}_{1,2,3} \right],
\end{aligned}
\end{equation*}
where $\text{perm}_{1,2,3}$ means we sum over all permutations of $p_1,p_2,p_3$. To this, we apply $\left( \frac{1}{2\pi i} \right)^4 \int_{p \in b_i} \int_{p_1 \in b_j} \int_{p_2 \in b_k} \int_{p_3 \in b_l}$, giving
\begin{equation}\label{equ:e1}
\begin{aligned}
& \frac{1}{2\pi i} \int_{p \in b_i} \sum_{a,b} \underset{q \to a}{Res} \underset{r \to b}{Res} \frac{dE_q(p)}{\omega(q)}\frac{dE_r(q)}{\omega(r)} \left[ \omega_j(\overline{q})\omega_k(r)\omega_l(\overline{r}) + \text{perm}_{j,k,l} \right] \\
& + \frac{1}{2\pi i} \int_{p \in b_i} \sum_{a,b} \underset{q \to a}{Res} \underset{r \to b}{Res} \frac{dE_q(p)}{\omega(q)}\frac{dE_r(\overline{q})}{\omega(r)} \left[ \omega_j(q)\omega_k(r)\omega_l(\overline{r}) + \text{perm}_{j,k,l} \right].
\end{aligned}
\end{equation}
Around each ramification point, choose coordinates $x,q$ as usual. Then $dx(q) = 2q dq$, $dy = y'(q)dq$ and we have:
\[
\underset{q \to b}{Res} \frac{dE_q(p) f(q)}{\omega(q)} = -\frac{1}{2} \underset{q \to b}{Res} \frac{B(q,p)f(q)}{dx(q)dy(q)} = -\frac{B(b,p)}{4y'(0)}f(b),
\]
where $f(q)$ is any local section of $K_S^2$, holomorphic in a neighbourhood of $b$. Using this, (\ref{equ:e1}) simplifies to:
\begin{equation}\label{equ:e2}
\begin{aligned}
& \frac{1}{4\pi i} \int_{p \in b_i} \sum_{a,b} \underset{q \to a}{Res} \; B(q,b)\frac{dE_q(p)}{\omega(q)} \omega_j(\overline{q}) \; \left(\frac{\omega_k(b)\omega_l(b)}{2y'(b)} \right) + \text{perm}_{j,k,l} \\
& + \frac{1}{4\pi i} \int_{p \in b_i} \sum_{a,b} \underset{q \to a}{Res} \; B(\overline{q},b)\frac{dE_q(p)}{\omega(q)} \omega_j(q) \; \left(\frac{\omega_k(b)\omega_l(b)}{2y'(b)} \right) + \text{perm}_{j,k,l}.
\end{aligned}
\end{equation}
Consider first the terms where $a \neq b$. Then $B(q,b)$ and $B(\overline{q},b)$ have no pole at $q = a$, so these terms give, on further simplification:
\begin{equation*}
\begin{aligned}
& \frac{1}{4\pi i} \int_{p \in b_i} \sum_{a \neq b} B(a,b) \frac{B(a,p)\omega_j(a)}{2y'(a)}  \left(\frac{\omega_k(b)\omega_l(b)}{2y'(b)} \right) + \text{perm}_{j,k,l} \\
&=  \frac{1}{2} \sum_{a \neq b} B(a,b) \left( \frac{\omega_i(a)\omega_j(a)}{2y'(a)} \right) \left(\frac{\omega_k(b)\omega_l(b)}{2y'(b)} \right) + \text{perm}_{j,k,l}  \\
&= \sum_{a \neq b} B(a,b) \left( \frac{\omega_i(a)\omega_j(a)}{2y'(a)} \right) \left(\frac{\omega_k(b)\omega_l(b)}{2y'(b)} \right) + \text{cyc}_{j,k,l} \\
\end{aligned}
\end{equation*}
where $\text{cyc}_{j,k,l}$ means a sum over cyclic permutations of $j,k,l$. Now let us consider the terms where $a = b$. In this case we must compute the residue as $q \to a$ of $\frac{1}{2\pi i } \int_{p \in b_i } B(q,a)\frac{dE_q(p)}{\omega(q)} \omega_j(\overline{q})$ and of $ \frac{1}{2\pi i } \int_{p \in b_i }  B(\overline{q},a)\frac{dE_q(p)}{\omega(q)} \omega_j(q)$. Both of these have poles of third order. We have the following expansions near $q=0$:
\begin{equation*}
\begin{aligned}
y(q) &= y(0) + y'(0)q + \frac{1}{2} y''(0) q^2 + \frac{1}{6} y'''(0)q^3 + \cdots \\
\omega(q) &= (y(q)-y(-q))2qdq = 4q^2( y'(0) + \frac{1}{6}y'''(0)q^2 + \cdots )dq \\
\omega_j(q) &= \left( \omega_j(0) + \omega'_j(0)q + \frac{1}{2}\omega''_j(0)q^2 + \frac{1}{6}\omega'''_j(0)q^3 + \cdots \right) dq \\
B(q,a) &= \left( \frac{1}{q^2} + \frac{1}{6}S_B(a) + \cdots \right) dq da
\end{aligned}
\end{equation*}
where $S_B(a)$ is the Bergman projective connection \cite{Fay0}. Also, we may compute the expansion of $\frac{1}{2\pi i} \int_{p \in b_i} dE_q(p)$:
\begin{equation*}
\begin{aligned}
\frac{1}{2\pi i} \int_{p \in b_i} dE_q(p) &= \frac{1}{2 \pi i} \int_{p \in b_i} \frac{1}{2} \int_{\xi = q}^{\overline{q}} B(\xi , p) \\
&= \frac{1}{2} \int_{\xi = q}^{\overline{q}} \omega_i(\xi)\\
&= \frac{1}{2} \int_{\xi =q}^{-q} \left( \omega_i(0) + \omega'_i(0)q + \frac{1}{2}\omega''_i(0)q^2 + \cdots \right) \\
&= -q \left ( \omega_i(0) + \frac{1}{6}\omega''_i(0)q^2 + \cdots \right).
\end{aligned}
\end{equation*}
Therefore, the expansion of $\frac{1}{2\pi i } \int_{p \in b_i } B(q,a)\frac{dE_q(p)}{\omega(q)} \omega_j(\overline{q})$ has the form
\begin{equation*}
\begin{aligned}
& \frac{1}{2\pi i } \int_{p \in b_i } B(q,a)\frac{dE_q(p)}{\omega(q)} \omega_j(\overline{q}) \\
&= \left( \frac{1}{q^2} + \frac{1}{6}S_B(a) + \cdots \right)\frac{ q \left ( \omega_i(0) + \frac{1}{6}\omega''_i(0)q^2 + \cdots \right) }{ 4q^2( y'(0) + \frac{1}{6}y'''(0)q^2 + \cdots ) } \left(\omega_j(0) - \omega'_j(0)q + \frac{1}{2}\omega''_j(0)q^2 + \cdots \right)dq da \\
&= \left( \frac{1}{q^2} + \frac{1}{6}S_B(a) + \cdots \right)\frac{ \left ( \omega_i(0) + \frac{1}{6}\omega''_i(0)q^2 + \cdots \right) }{ 4q( y'(0) + \frac{1}{6}y'''(0)q^2 + \cdots ) } \left(\omega_j(0) - \omega'_j(0)q + \frac{1}{2}\omega''_j(0)q^2 + \cdots \right)dq da.
\end{aligned}
\end{equation*}
Adding this to the corresponding expansion for $\frac{1}{2\pi i } \int_{p \in b_i } B(\overline{q},a)\frac{dE_q(p)}{\omega(q)} \omega_j(q)$, we get
\begin{equation*}
\begin{aligned}
& \frac{1}{2\pi i } \int_{p \in b_i } B(q,a)\frac{dE_q(p)}{\omega(q)} \omega_j(\overline{q}) + \frac{1}{2\pi i } \int_{p \in b_i } B(\overline{q},a)\frac{dE_q(p)}{\omega(q)} \omega_j(q) \\
&= \left( \frac{1}{q^2} + \frac{1}{6}S_B(a) + \cdots \right)\frac{ \left ( \omega_i(0) + \frac{1}{6}\omega''_i(0)q^2 + \cdots \right) }{ 2q( y'(0) + \frac{1}{6}y'''(0)q^2 + \cdots ) } \left(\omega_j(0) + \frac{1}{2}\omega''_j(0)q^2 + \cdots \right)dq da.
\end{aligned}
\end{equation*}
The coefficient of $\frac{dq}{q}$ in this is:
\[
\frac{1}{12 y'(0)}\left( S_B(a) - \frac{y'''(0)}{y'(0)} \right) \omega_i(0)\omega_j(0) + \frac{1}{2y'(0)}\left( \frac{1}{2}\omega_i(0) \omega''_j(0) + \frac{1}{6}\omega''_i(0)\omega_j(0) \right).
\]
Therefore, the $a=b$ terms in (\ref{equ:e2}) are given by:
\begin{equation*}
\begin{aligned}
& \frac{1}{2} \sum_a  \frac{1}{24 y'(a)^2}\left( S_B(a) - \frac{y'''(a)}{y'(a)} \right) \omega_i(a)\omega_j(a)\omega_k(a)\omega_l(a) + \text{perm}_{j,k,l}\\
& + \frac{1}{2} \sum_a \frac{1}{4y'(a)^2}\left( \frac{1}{2}\omega_i(a) \omega''_j(a)\omega_k(a)\omega_l(a) + \frac{1}{6}\omega''_i(a)\omega_j(a)\omega_k(a)\omega_l(a) \right) + \text{perm}_{j,k,l} \\
&= \sum_a  \frac{1}{8 y'(a)^2}\left( S_B(a) - \frac{y'''(a)}{y'(a)} \right) \omega_i(a)\omega_j(a)\omega_k(a)\omega_l(a) \\
& \quad \quad + \sum_a \frac{1}{8y'(a)^2}\left( \omega_i''(a) \omega_j(a)\omega_k(a)\omega_l(a) \right) + \text{cyc}_{i,j,k,l}.
\end{aligned}
\end{equation*}
Putting this all together we have shown:
\begin{theorem}
The second derivatives of the period matrix are given by:
\begin{equation}\label{equ:2ndderiv}
\begin{aligned}
\partial_i \partial_j \tau_{kl} &= \left( \frac{1}{2\pi i} \right)^3 \int_{p \in b_i} \int_{p_1 \in b_j} \int_{p_2 \in b_k} \int_{p_3 \in b_l} W^{(0)}_4(p,p_1,p_2,p_3) \\
&= 2\pi i \sum_{a \neq b} B(a,b) \left( \frac{ \omega_i(a)\omega_j(a)}{2y'(a)} \right) \left( \frac{ \omega_k(b)\omega_l(b)}{2y'(b)} \right) + \text{cyc}_{j,k,l} \\
& \quad \quad + 2\pi i \sum_a  \frac{1}{8 y'(a)^2}\left( S_B(a) - \frac{y'''(a)}{y'(a)} \right) \omega_i(a)\omega_j(a)\omega_k(a)\omega_l(a) \\
& \quad \quad + 2 \pi i \sum_a \frac{1}{8y'(a)^2}\left( \omega_i''(a) \omega_j(a)\omega_k(a)\omega_l(a) \right) + \text{cyc}_{i,j,k,l}.
\end{aligned}
\end{equation}
\end{theorem}

\begin{remark}
Let us verify that the right hand side of (\ref{equ:2ndderiv}) is independent of the choice of local coordinates $q,x$ satisfying $x = q^2$, as it must be, since $\partial_i \partial_j \tau_{kl}$ is independent of such choices. Consider a change of variables $\hat{x} = h(x)$, $\hat{q} = f(q)$, where $x = q^2$ and $\hat{x}= \hat{q}^2$. If $h(x) = h'x + \frac{1}{2}h'' x^2 + \frac{1}{6} h''' x^3 + \cdots $ and $f(q) = f' q + \frac{1}{2} f'' q^2 + \frac{1}{6} f''' q^3 + \cdots$ then the relation $\hat{x} = \hat{q}^2$ implies $f''(0) = 0$. Let $S_B(a), \hat{S}_B(a)$ denote the Bergman projective connection in the $q$ and $\hat{q}$-coordinates. The property of being a projective connection means
\[
S_B(a) = (f')^2 \hat{S}_B(a) + \left( \frac{f'''}{f'} - \frac{3}{2}\frac{f''^2}{f'^2} \right),
\]
where $\left( \frac{f'''}{f'} - \frac{3}{2}\frac{f''^2}{f'^2} \right) = S(f)$ is the Schwarzian derivative of $f$ at $q=0$. From $\theta = y(q)dx = \hat{y}(\hat{q})d\hat{x}$, one finds 
\begin{equation}\label{equ:change1}
y'(0) = \hat{y}'(0) (f')^3
\end{equation}
and
\[
\frac{y'''(0)}{y'(0)} = \frac{\hat{y}'''(0)}{\hat{y}'(0)} (f')^2 + 5 \frac{f'''}{f'}.
\]
Then since $f'' = 0$, we get
\begin{equation}\label{equ:change2}
\left( S_B(a) - \frac{y'''(0)}{y'(0)} \right) = (f')^2 \left( \hat{S}_B(a) - \frac{\hat{y}'''(0)}{\hat{y}'(0)} \right) - 4 \frac{f'''}{f'}.
\end{equation}
But using $\omega_i(q)dq = \hat{\omega}_i(\hat{q})d\hat{q}$, we also find
\begin{align}
\omega_i(a) &= (f') \hat{\omega}_i(a) \label{equ:change3}\\
\omega''_i(a) &= (f')^3 \hat{\omega}''_i(a) + \hat{\omega}_i(a) f'''. \label{equ:change4}
\end{align}
Substituting (\ref{equ:change1})-(\ref{equ:change4}) into (\ref{equ:2ndderiv}), we see that the result is coordinate independent.
\end{remark}


\bibliographystyle{amsplain}

\begin{thebibliography}{99}
\bibitem{bal}D. Balduzzi, Donagi-Markman cubic for Hitchin systems. {\em Math. Res. Lett.} {\bf 13} (2006), no. 5-6, 923-933. 
\bibitem{bnr}A. Beauville, M. S. Narasimhan, S. Ramanan, Spectral curves and the generalised theta divisor. {\em J. Reine Angew. Math.} {\bf 398} (1989), 169-179.
\bibitem{cfg}S. Cecotti, S. Ferrara, L. Girardello, Geometry of type II superstrings and the moduli of superconformal field theories. {\em Internat. J. Modern Phys. A} {\bf 4} (1989), no. 10, 2475-2529.
\bibitem{dw}R. Donagi, E. Witten, Supersymmetric Yang-Mills theory and integrable systems. {\em Nuclear Phys. B} {\bf 460} (1996), no. 2, 299-334. 
\bibitem{dm}R. Donagi, E. Markman, {\em Cubics, integrable systems, and Calabi-Yau threefolds}, Proceedings of the Hirzebruch 65 Conference on Algebraic Geometry (Ramat Gan, 1993), 199-221, Israel Math. Conf. Proc., {\bf 9}, Bar-Ilan Univ., Ramat Gan, 1996. 
\bibitem{dumu}O. Dumitrescu, M. Mulase, Quantum curves for Hitchin fibrations and the Eynard-Orantin theory. {\em Lett. Math. Phys.} {\bf 104} (2014), no. 6, 635-671. 
\bibitem{EO}B. Eynard, N. Orantin, Invariants of algebraic curves and topological expansion. {\em Commun. Number Theory Phys.} {\bf 1} (2007), no. 2, 347-452. 
\bibitem{Fay0}J. Fay, Theta functions on Riemann surfaces. Lecture Notes in Mathematics, Vol. 352. Springer-Verlag, Berlin-New York, 1973. iv+137 pp. 
\bibitem{Fay}J. Fay, {\em Kernel functions, analytic torsion, and moduli spaces}. Mem. Amer. Math. Soc. {\bf 96} (1992), no. 464, vi+123 pp. 
\bibitem{Freed}D. S. Freed, Special K\"ahler manifolds. {\em Comm. Math. Phys.} {\bf 203} (1999), no. 1, 31-52. 
\bibitem{gmn}D. Gaiotto, G. Moore, A. Neitzke, Four-dimensional wall-crossing via three-dimensional field theory. {\em Comm. Math. Phys.} {\bf 299} (2010), no. 1, 163-224. 
\bibitem{gri}P. A. Griffiths, Periods of integrals on algebraic manifolds. II. Local study of the period mapping. {\em Amer. J. Math.} {\bf 90} (1968) 805-865.
\bibitem{hhp}C. Hertling, L. Hoevenaars, H. Posthuma, Frobenius manifolds, projective special geometry and Hitchin systems. {\em J. Reine Angew. Math.} {\bf 649} (2010), 117-165. 
\bibitem{hit1}N. J. Hitchin, The self-duality equations on a Riemann surface. {\em Proc. London Math. Soc.} (3) {\bf 55} (1987), no. 1, 59-126. 
\bibitem{hit3}N. J. Hitchin, Stable bundles and integrable systems. {\em Duke Math. J.} {\bf 54} (1987), no. 1, 91-114.
\bibitem{Hit2}N. J. Hitchin, The moduli space of complex Lagrangian submanifolds. {\em Asian J. Math.} {\bf 3} (1999), no. 1, 77-91. 
\bibitem{kp}A. Kouvidakis, T. Pantev, The automorphism group of the moduli space of semistable vector bundles. {\em Math. Ann.} {\bf 302} (1995), no. 2, 225-268.
\bibitem{nit}N. Nitsure, Moduli space of semistable pairs on a curve. {\em Proc. London Math. Soc.} (3) {\bf 62} (1991), no. 2, 275-300. 
\bibitem{sim}C. T. Simpson, Higgs bundles and local systems. {\em Inst. Hautes \'Etudes Sci. Publ. Math.} No. {\bf 75} (1992), 5-95. 
\bibitem{Voi}C. Voisin, Sur la stabilit\'e des sous-vari\'et\'es lagrangiennes des vari\'et\'es symplectiques holomorphes, in {\em Complex projective geometry} (Trieste 1989/Bergen 1989), 294-303. London Math. Soc. Lecture Notes {\bf 179}, Cambridge Univ. Press, Cam- bridge (1992).
\end{thebibliography}

\end{document}